\documentclass[a4paper,12pt]{amsart}
\usepackage{fullpage}
\newtheorem{thm}{Theorem}
\newtheorem{prop}{Proposition}
\newtheorem{cor}{Corollary}
\DeclareSymbolFont{script}{U}{eus}{m}{n}
\DeclareMathSymbol{\Wedge}{0}{script}{"5E}
\newcommand{\Rho}{\mathrm{P}}
\def\dddots{\mathinner{\mkern1mu\raise7\p@
\vbox{\kern7\p@\hbox{.}}\mkern2mu
\raise4\p@\hbox{.}\mkern2mu\raise\p@\hbox{.}\mkern1mu}}
\begin{document}

\title[Killing tensors]
{Killing tensors on complex projective space}
\author[M.G.~Eastwood]{Michael Eastwood}
\address{\hskip-\parindent
School of Mathematical Sciences\\
University of Adelaide\\ 
SA 5005\\ 
Australia}
\email{meastwoo@gmail.com}
\begin{abstract}
The Killing tensors of arbitrary rank on complex projective space with its
Fubini-Study metric are determined and it is shown that these spaces are 
generated by the Killing fields.
\end{abstract}
\thanks{This work was supported by Australian Research Council (Discovery 
Project DP190102360).}

\maketitle
\section{Introduction}
On a smooth Riemannian manifold, a {\em Killing tensor\/} of rank $k$ is a
smooth section of the bundle $\bigodot^k\!\Wedge^1$ lying in the kernel of the
differential operator
\begin{equation}\label{killing_tensor_operator}
\textstyle\nabla:\bigodot^k\!\Wedge^1\to\bigodot^{k+1}\!\Wedge^1
\end{equation}
given by $\sigma_{bc\cdots d}\mapsto\nabla_{(a}\sigma_{bc\cdots d)}$, where 
$\nabla_a$ is the Levi-Civita connection and round brackets mean to take the 
symmetric part. Here, and throughout this article, we are employing 
{\em abstract indices\/} in the sense of~\cite{OT} so that, for example, the 
Riemann curvature tensor $R_{abcd}$ is characterised by 
\begin{equation}\label{curvature_tensor}
(\nabla_a\nabla_b-\nabla_b\nabla_a)X^c=R_{ab}{}^c{}_dX^d\end{equation}
and we are using the metric $g_{ab}$ to {\em raise and lower indices\/} so that
$R_{abcd}\equiv g_{ce}R_{ab}{}^e{}_d$.  In particular, a {\em Killing field} is
a vector field $X^a$ so that $X_b=g_{ab}X^a$ is a Killing tensor of rank~$1$.
The effect of the Lie derivative on the metric
${\mathcal{L}}_Xg_{ab}=2\nabla_{(a}X_{b)}$ interprets Killing fields as 
infinitesimal Riemannian symmetries. 

The Killing tensors on the round sphere are well-known.  That the
theory works out so well can be attributed to the following facts:
\begin{itemize}
\item the differential operator (\ref{killing_tensor_operator}) is projectively
invariant (for $\sigma_{bc\cdots d}$ of weight~$2k$), 
\item the round sphere is projectively flat.
\end{itemize}
The upshot of these observations is that the group
${\mathrm{SL}}(n+1,{\mathbb{R}})$, as projective symmetries of the $n$-sphere, 
acts on the space of Killing tensors of rank $k$ and it turns out that this 
action realises certain {\em irreducible\/} representations 
of~${\mathrm{SL}}(n+1,{\mathbb{R}})$, which may be identified as follows 
(using the Dynkin diagram conventions of~\cite{Beastwood}). 
\begin{thm}\label{killing_tensors_on_the_sphere}
The space of Killing tensors of rank $k$ on the round $n$-sphere is
naturally identified as the irreducible representation $\begin{picture}(110,14)
\put(5,2){\makebox(0,0){$\bullet$}}
\put(25,2){\makebox(0,0){$\bullet$}}
\put(45,2){\makebox(0,0){$\bullet$}}
\put(65,2){\makebox(0,0){$\bullet$}}
\put(5,2){\line(1,0){70}}
\put(85,2){\makebox(0,0){$\cdots$}}
\put(95,2){\line(1,0){10}}
\put(105,2){\makebox(0,0){$\bullet$}}
\put(5,9){\makebox(0,0){$\scriptstyle 0$}}
\put(25,9){\makebox(0,0){$\scriptstyle k$}}
\put(45,9){\makebox(0,0){$\scriptstyle 0$}}
\put(65,9){\makebox(0,0){$\scriptstyle 0$}}
\put(105,9){\makebox(0,0){$\scriptstyle 0$}}
\end{picture}$ of ${\mathrm{SL}}(n+1,{\mathbb{R}})$.
\end{thm}
There are formul{\ae} (e.g.~\cite{FH,King}) for the dimensions of such
representations:
\begin{cor}
The space of Killing tensors of rank $k$ on the round $n$-sphere is
a finite-dimensional vector space of dimension
\begin{equation}\label{DTT}\frac{(n+k-1)!(n+k)!}{k!(k+1)!(n-1)!n!}
\end{equation}
\end{cor}
Alternatively, there are convenient individual formul{\ae} for fixed rank, 
for example:
\begin{cor} The space of Killing tensors of rank $3$ on the round $n$-sphere 
is a finite-dimensional vector space of dimension
$n(n+1)^2(n+2)^2(n+3)/144$.
\end{cor}
If $X^a$ is a Killing field on any Riemannian manifold, then 
$$\underbrace{X_{(b}X_c\cdots X_{d)}}_k$$ 
is a Killing tensor of rank~$k$.  On the round sphere,
Theorem~\ref{killing_tensors_on_the_sphere} gives the following.  
\begin{cor}
The Killing tensors on the round sphere are generated by the Killing fields.
\end{cor}
\begin{proof} It is because ${\mathrm{SL}}(n+1,{\mathbb{R}})$ acts 
{\em irreducibly\/} on the Killing tensors of fixed rank.
\end{proof}
There are nowadays many proofs of Theorem~\ref{killing_tensors_on_the_sphere}.
We may as well, for example, consider the space of Killing tensors, or even 
{\em generalised Killing tensors\/}~\cite{NP}
$$\textstyle\sigma_{c\cdots d}\in\Gamma({\mathbb{R}}^n,\bigodot^k\!\Wedge^1)
\enskip\mbox{ such that}\enskip
\underbrace{\nabla_{(a}\cdots\nabla_b}_{j+1}\sigma_{c\cdots d)}=0,$$
as a representation of the corresponding Lie algebra
${\mathfrak{sl}}(n+1,{\mathbb{R}})$ acting by vector fields as the
infinitesimal projective symmetries of~${\mathbb{R}}^n$.  By identifying its
highest weight vector, it is shown in~\cite{E-treves} that this representation
is irreducible and, with notation of~\cite{Beastwood}, it is
$$\begin{picture}(110,14)
\put(5,2){\makebox(0,0){$\bullet$}}
\put(25,2){\makebox(0,0){$\bullet$}}
\put(45,2){\makebox(0,0){$\bullet$}}
\put(65,2){\makebox(0,0){$\bullet$}}
\put(5,2){\line(1,0){70}}
\put(85,2){\makebox(0,0){$\cdots$}}
\put(95,2){\line(1,0){10}}
\put(105,2){\makebox(0,0){$\bullet$}}
\put(5,9){\makebox(0,0){$\scriptstyle j$}}
\put(25,9){\makebox(0,0){$\scriptstyle k$}}
\put(45,9){\makebox(0,0){$\scriptstyle 0$}}
\put(65,9){\makebox(0,0){$\scriptstyle 0$}}
\put(105,9){\makebox(0,0){$\scriptstyle 0$}}
\end{picture}.$$
A more standard approach, however, is to realise the Killing tensors as the
covariant constant sections of an appropriate projectively invariant vector
bundle with connection.  As explained in~\cite{tractors}, the simplest such
bundle is due to T.Y. Thomas~\cite{T} and, as sketched in~\cite{E-IMA_notes},
is equivalent to the well-known projective Cartan connection~\cite{C}.  In the
following section, we shall present these {\em tractor bundles\/}, for
simplicity only on the round sphere, where the tractor connections are flat.
The general projectively invariant tractor bundles with connection prolonging
the Killing tensor equation are constructed in~\cite{GL}.

Unfortunately, it is a theorem of Beltrami~\cite{B} that the only projectively
flat metric connections occur for metrics of constant sectional curvature.  In
particular, the Fubini-Study connection on ${\mathbb{CP}}_n$ is not
projectively flat for $n\geq 2$.  Fortunately, there are replacement 
{\em tractor connections\/} in K\"ahler geometry, which are {\em symplectically
flat\/}~\cite{ES} for metrics of constant holomorphic sectional curvature.  As
explained in~\S\ref{Kaehlerian_tractors}, there are parallel flat subbundles
whose covariant constant sections correspond to Killing tensors.  We arrive at
the following identification of Killing tensors on~${\mathbb{CP}}_n$ 
(Theorem~\ref{killing_tensors_on_cpn} below).

Let ${\mathbb{T}}$ denote the standard representation of~${\mathrm{SU}}(n+1)$.
More precisely, let us consider ${\mathbb{T}}$ as a real vector space of
dimension $2n+2$ equipped with the following features:
\begin{itemize}
\item a positive definite symmetric form $g_{\alpha\beta}$,
\item a compatible complex structure, viewed as an orthogonal endomorphism 
$$\Omega_\alpha\mapsto J_\alpha{}^\beta\Omega_\beta\enskip
\mbox{such that}\enskip J_\alpha{}^\beta 
J_\beta{}^\gamma=-\delta_\alpha{}^\gamma,$$ 
where $\delta_\alpha{}^\beta$ is the identity tensor:
$\Omega_\alpha=\delta_\alpha{}^\beta\Omega_\beta,\;\forall\Omega_\alpha$,
\item a nondegenerate skew form $J_{\alpha\beta}$ obtaining by lowering an 
index: $J_{\alpha\beta}=J_\alpha{}^\gamma g_{\beta\gamma}$.
\end{itemize}
Any two of $g_{\alpha\beta}$, $J_\alpha{}^\beta$,
$J_{\alpha\beta}$ determine the third and ${\mathrm{SU}}(n+1)$ is consequently
the intersection of any pair of
$${\mathrm{SO}}(2n+2)\qquad{\mathrm{SL}}(n+1,{\mathbb{C}})\qquad
{\mathrm{Sp}}(2n+2,{\mathbb{R}})$$
as subgroups of ${\mathrm{SL}}(2n+2n,{\mathbb{R}})$.  Now, to identify the
Killing tensors on~${\mathbb{CP}}_n$, we freely employ Young tableau as
in~\cite{FH,OT}.
\begin{thm}\label{killing_tensors_on_cpn}
The space of Killing tensors of rank $k$ on ${\mathbb{CP}}_n$ with its
Fubini-Study metric may be identified as the following vector space
\begin{equation}\label{key_vector_space}
\left\{\Sigma_{\alpha\beta\cdots\eta}\in
\begin{picture}(85,8)(0,5)
\put(0,0){\line(0,1){20}}
\put(10,0){\line(0,1){20}}
\put(20,0){\line(0,1){20}}
\put(50,0){\line(0,1){20}}
\put(60,0){\line(0,1){20}}
\put(0,0){\line(1,0){60}}
\put(0,10){\line(1,0){60}}
\put(0,20){\line(1,0){60}}
\put(36,4.5){\makebox(0,0){$\cdots$}}
\put(36,14.5){\makebox(0,0){$\cdots$}}
\put(72,10){\makebox(0,0){\large$({\mathbb{T}})$}}
\put(30,-9.5){\makebox(0,0){$\underbrace{\hspace{58pt}}_k$}}
\end{picture}\mbox{\Large$\mid$}\begin{tabular}{l} 
$\Sigma_{\alpha\beta\cdots\eta}$ is totally trace-free with respect to 
$J^{\alpha\beta}$\\ 
and $(J\Sigma)_{\alpha\beta\cdots\eta}=0$
\end{tabular}
\right\},\end{equation}

\bigskip\noindent where $J\Sigma$ denotes the `derivation' action
\begin{equation}\label{derivation_action}
\Sigma_{\alpha\beta\cdots\eta}\mapsto
J_\alpha{}^\lambda\Sigma_{\lambda\beta\cdots\eta}
+J_\beta{}^\mu\Sigma_{\alpha\mu\cdots\eta}+\cdots
+J_\eta{}^\nu\Sigma_{\alpha\beta\cdots\nu}.\end{equation}
\end{thm}

A significant difference between this theorem and the corresponding
Theorem~\ref{killing_tensors_on_the_sphere} on the sphere is that the space of
Killing tensors is no longer irreducible for any hidden symmetry group.  A
potential hidden symmetry group is ${\mathrm{Sp}}(2n+2,{\mathbb{R}})$ acting on
the sphere $S^{2n+1}$ but this action does not descend under the natural
projection $S^{2n+1}\to{\mathbb{CP}}_n$.  The best one can do in this regard is
to realise the space of rank $k$ Killing tensors as the kernel of
$$J:\underbrace{\begin{picture}(110,14)(0,-2)
\put(5,2){\makebox(0,0){$\bullet$}}
\put(25,2){\makebox(0,0){$\bullet$}}
\put(45,2){\makebox(0,0){$\bullet$}}
\put(65,2){\makebox(0,0){$\cdots$}}
\put(85,2){\makebox(0,0){$\bullet$}}
\put(105,2){\makebox(0,0){$\bullet$}}
\put(5,2){\line(1,0){50}}
\put(75,2){\line(1,0){10}}
\put(85,1){\line(1,0){20}}
\put(85,3){\line(1,0){20}}
\put(95,2){\makebox(0,0){$\langle$}}
\put(5,10){\makebox(0,0){$\scriptstyle 0$}}
\put(25,10){\makebox(0,0){$\scriptstyle k$}}
\put(45,10){\makebox(0,0){$\scriptstyle 0$}}
\put(85,10){\makebox(0,0){$\scriptstyle 0$}}
\put(105,10){\makebox(0,0){$\scriptstyle 0$}}
\end{picture}}_{\mbox{\scriptsize$n+1$ nodes}}\longrightarrow
\begin{picture}(110,14)(0,-2)
\put(5,2){\makebox(0,0){$\bullet$}}
\put(25,2){\makebox(0,0){$\bullet$}}
\put(45,2){\makebox(0,0){$\bullet$}}
\put(65,2){\makebox(0,0){$\cdots$}}
\put(85,2){\makebox(0,0){$\bullet$}}
\put(105,2){\makebox(0,0){$\bullet$}}
\put(5,2){\line(1,0){50}}
\put(75,2){\line(1,0){10}}
\put(85,1){\line(1,0){20}}
\put(85,3){\line(1,0){20}}
\put(95,2){\makebox(0,0){$\langle$}}
\put(5,10){\makebox(0,0){$\scriptstyle 0$}}
\put(25,10){\makebox(0,0){$\scriptstyle k$}}
\put(45,10){\makebox(0,0){$\scriptstyle 0$}}
\put(85,10){\makebox(0,0){$\scriptstyle 0$}}
\put(105,10){\makebox(0,0){$\scriptstyle 0$}}
\end{picture},$$
where $J$ is acting as above~(\ref{derivation_action}).

Nevertheless, in \S\ref{Kaehlerian_tractors} we shall use
Theorem~\ref{killing_tensors_on_cpn} to determine the various dimensions of the
spaces of Killing tensors on~${\mathbb{CP}}_n$.  For example, we obtain:
\begin{cor} The space of Killing tensors of rank $3$ on ${\mathbb{CP}}_n$ is a
finite-dimensional vector space of dimension $n(n+1)^2(5n^3+25n^2+35n+24)/36$.
\end{cor}
In \S\ref{generation} we shall gather enough information to
deduce, from Theorem~\ref{killing_tensors_on_cpn}, the following.
\begin{cor}\label{killing_fields_generate}
The Killing tensors on ${\mathbb{CP}}_n$ are generated by the 
Killing fields.
\end{cor}

I have been informed by Vladimir Matveev that he and Yuri Nikolayevsky have 
recently found the Killing tensors of rank~$2$ on both ${\mathbb{CP}}_n$ 
and~${\mathbb{HP}}_n$. 

I would like to thank Federico Costanza, Thomas Leistner, and Benjamin McMillan
for many useful conversations concerning this work.

\section{The Killing connection}
Suppose $\nabla_a$ is a torsion-free connection on a smooth manifold~$M$.  
Following~\cite{CELM}, we
may define the {\em Killing connection\/} on the vector bundle
$E\equiv\Wedge^1\oplus\Wedge^2$ by
\begin{equation}\label{the_killing_connection}
\nabla_b\left[\!\begin{array}{c}\sigma_c\\ \mu_{cd}\end{array}\!\right]
=\left[\!\begin{array}{c}\nabla_b\sigma_c-\mu_{bc}\\ 
\nabla_a\mu_{cd}-R_{cd}{}^e{}_b\sigma_e\end{array}\!\right]\end{equation} 
where $R_{ab}{}^c{}_d$ is defined by (\ref{curvature_tensor}).
\begin{prop}[cf.~\cite{K}]\label{prolongation} On any Riemannian manifold, the
differential operator
$$\Wedge^1\to E\quad\mbox{given by}\enskip
\sigma_b\mapsto
\left[\!\begin{array}{c}\sigma_b\\ \nabla_{[b}\sigma_{c]}\end{array}\!\right]$$
induces an isomorphism 
\begin{equation}\label{prolongation_isomorphism}
\{\sigma_b\in\Gamma(\Wedge^1)\mid\nabla_{(a}\sigma_{b)}=0\}
\cong
\{\Sigma\in\Gamma(E)\mid\nabla_a\Sigma=0\}.\end{equation}
\end{prop}
\begin{proof} For a covariant constant section of~$E$, already the 
first line of the connection tells us that $\nabla_b\sigma_c$ is skew, which 
is precisely saying that $\nabla_{(b}\sigma_{c)}=0$. Conversely, if 
$\nabla_{(b}\sigma_{c)}=0$, then we may set $\mu_{bc}=\nabla_b\sigma_c$, which 
is skew and manifestly closed. Therefore,
\begin{equation}\label{closure}
\nabla_a\mu_{bc}=\nabla_c\mu_{ba}-\nabla_b\mu_{ca}
=\nabla_c\nabla_b\sigma_a-\nabla_b\nabla_c\sigma_a=R_{bc}{}^d{}_a\sigma_d,
\end{equation}
as required. \end{proof}

This reasoning is typical of `prolongation' in which one introduces variables
for unknown derivatives and then attempts to determine the derivatives of these
new variables as differential consequences of the original equations.  We can
say that the connection (\ref{the_killing_connection}) is a `prolongation
connection' for the Killing operator.  Beware, however, that
in~\S\ref{Kaehlerian_tractors} we shall find on ${\mathbb{CP}}_n$ another
connection with the same property.

For the curvature of the Killing connection, a straightforward calculation
shows that
$$(\nabla_a\nabla_b-\nabla_b\nabla_a)
\left[\!\begin{array}{c}\sigma_c\\ \mu_{cd}\end{array}\!\right]
=\left[\!\begin{array}{c}0\\ 
2R_{ab}{}^e{}_{[c}\mu_{d]e}+2R_{cd}{}^e{}_{[a}\mu_{b]e}
-(\nabla^eR_{abcd})\sigma_e
\end{array}\!\right].$$
Hence, the Killing connection is flat on the unit sphere where
$R_{abcd}=g_{ac}g_{bd}-g_{bc}g_{ad}$.  It follows that the space of Killing
fields on $S^n$ may be identified with the fibre of $E$ at any point and, in
particular, its dimension is $n(n+1)/2$ (in agreement with (\ref{DTT})
for~$k=1$).

On ${\mathbb{CP}}_n$ we may scale the
Fubini-Study metric $g_{ab}$ and K\"ahler form $J_{ab}$ so that
\begin{equation}\label{cpn_curvature}
R_{abcd}=g_{ac}g_{bd}-g_{bc}g_{ad}+J_{ac}J_{bd}-J_{bc}J_{ad}+2J_{ab}J_{cd}
\end{equation}
and so $\nabla^eR_{abcd}=0$ whilst
\begin{equation}\label{mu_is_type_one_one}
R_{ab}{}^e{}_{[c}\mu_{d]e}+R_{cd}{}^e{}_{[a}\mu_{b]e}
=J_{bc}\xi_{ad}-J_{ac}\xi_{bd}-J_{bd}\xi_{ac}+J_{ad}\xi_{bc}
-2J_{ab}\xi_{cd}-2J_{cd}\xi_{ab},\end{equation}
where $\xi_{ab}\equiv J_{[a}{}^c\mu_{b]c}$. So the Killing connection is not 
flat but instead has a flat subbundle
$$F\equiv\begin{array}{c}\Wedge^1\\[-3pt] \oplus\\[-1pt] 
K\end{array}\enskip\mbox{where}\enskip
K\equiv\{\mu_{cd}\mid R_{ab}{}^e{}_{[c}\mu_{d]e}+R_{cd}{}^e{}_{[a}\mu_{b]e}=0\}
=\{\mu_{cd}\mid J_{[a}{}^c\mu_{b]c}=0\}=\Wedge_{\mathbb{R}}^{1,1},$$
where $\Wedge_{\mathbb{R}}^{1,1}$ denotes the real $2$-forms of type~$(1,1)$. 
As explained in~\cite{CELM}, it is a general feature of the Killing connection 
on any Riemannian locally symmetric space that $F$ is parallel and, since 
${\mathbb{CP}}_n$ is simply-connected, we conclude 
that the Killing fields on ${\mathbb{CP}}_n$ may be identified with the fibre 
of $F$ at any chosen `basepoint.' Of course, this is consistent with the 
well-known identification of the symmetry algebra as
$${\mathfrak{su}}(n+1)={\mathfrak{u}}(n)\oplus{\mathfrak{m}}$$
where ${\mathfrak{m}}$ is a vector space of dimension $=\dim{\mathbb{CP}}_n=2n$
and $\dim{\mathfrak{u}}(n)$ = $n^2$.  In particular, the vector space of
Killing fields on ${\mathbb{CP}}_n$ has dimension $n(n+2)$.

\section{Riemannian tractor bundles}
On any Riemannian manifold with metric $g_{ab}$, let us introduce the vector 
bundle
\begin{equation}\label{R_tractors}
{\mathcal{T}}\equiv
\begin{array}{c}\Wedge^0\\[-3pt] \oplus\\[-1pt] \Wedge^1\end{array}\enskip
\mbox{with connection}\enskip
\nabla_a\left[\!\begin{array}{c}\sigma\\ \mu_b\end{array}\!\right]
\!\equiv\!\left[\!\begin{array}{c}\nabla_a\sigma-\mu_a\\ 
\nabla_a\mu_b+g_{ab}\sigma\end{array}\!\right],\end{equation}
where $\nabla_a\mu_b$ denotes the Levi-Civita connection on $1$-forms.
\begin{prop}\label{R-prolong}
The differential operator
$$\Wedge^0\to{\mathcal{T}}\quad\mbox{given by}\enskip
\sigma\mapsto
\left[\!\begin{array}{c}\sigma\\ \nabla_b\sigma\end{array}\!\right]$$
induces an isomorphism 
$$\{\sigma\in\Gamma(\Wedge^0)\mid\nabla_a\nabla_b\sigma+g_{ab}\sigma=0\}
\cong\{\Sigma_\beta\in\Gamma({\mathcal{T}})\mid\nabla_a\Sigma_\beta=0\}.$$
\end{prop}
\begin{proof} This is simply a matter of unpacking the meaning of 
$\nabla_a\Sigma_\beta=0$. In effect, one is setting $\mu_b=\nabla_b\sigma$ and 
rewriting the differential equation $\nabla_a\nabla_b\sigma+g_{ab}\sigma=0$ 
accordingly. (This is an especially simple case of prolongation.)
\end{proof}
\begin{prop} The curvature of the Riemannian tractor connection is given by 
$$(\nabla_a\nabla_b-\nabla_b\nabla_a)
\left[\!\begin{array}c\sigma\\ \mu_c\\ \end{array}\!\right]
=\left[\!\begin{array}{c}0\\ 
g_{bc}\mu_a-g_{ac}\mu_b-R_{ab}{}^d{}_c\mu_d\\ \end{array}\!\right].$$
\end{prop}
\begin{proof}
We compute
$$\nabla_a\nabla_b\left[\!\begin{array}c\sigma\\ \mu_c \end{array}\!\right]
=\left[\!\begin{array}c
\nabla_a(\nabla_b\sigma-\mu_b)-(\nabla_b\mu_a+g_{ba}\sigma)\\ 
\nabla_a(\nabla_b\mu_c+g_{bc}\sigma)+g_{ac}(\nabla_b\sigma-\mu_b)
\\ \end{array}\!\right]$$
so
$$(\nabla_a\nabla_b-\nabla_b\nabla_a)
\left[\!\begin{array}c\sigma\\ \mu_c\\ \end{array}\!\right]
=\left[\!\begin{array}c0\\ 
(\nabla_a\nabla_b-\nabla_b\nabla_a)\mu_c+g_{bc}\mu_a-g_{ac}\mu_b
\end{array}\!\right],$$
as required.\end{proof}
\begin{cor} The tractor connection \eqref{R_tractors} is flat on the unit 
sphere~$S^n$.
\end{cor}
\begin{proof} The Riemann curvature tensor on the unit $n$-sphere
is $R_{abcd}=g_{ac}g_{bd}-g_{bc}g_{ad}$.
\end{proof}
\begin{prop} The induced tractor connection on $\Wedge^2{\mathcal{T}}$ is 
given by 
\begin{equation}\label{killing_connection_on_sphere}
\Wedge^2{\mathcal{T}}=
\begin{array}{c}\Wedge^1\\[-3pt] \oplus\\[-1pt] \Wedge^2\end{array}\enskip
\mbox{with}\enskip
\nabla_a\left[\!\begin{array}{c}\sigma_b\\ \mu_{bc}\end{array}\!\right]
\!=\!\left[\!\begin{array}{c}\nabla_a\sigma_b-\mu_{ab}\\ 
\nabla_a\mu_b-g_{ac}\sigma_b+g_{ab}\sigma_c\end{array}\!\right].\end{equation}
\end{prop}
\begin{proof} It suffices to verify this for tractors of the form
$\Sigma_{[\alpha}\tilde\Sigma_{\beta]}$.  This verification is left to the
reader.
\end{proof}
\begin{cor} The connection \eqref{killing_connection_on_sphere} is flat on the 
unit $n$-sphere~$S^n$.
\end{cor}
\begin{prop}\label{killing_fields_prolonged}The differential operator
$$\Wedge^1\to\Wedge^2{\mathcal{T}}\quad\mbox{given by}\enskip
\sigma_b\mapsto
\left[\!\begin{array}{c}\sigma_b\\ \nabla_{[b}\sigma_{c]}\end{array}\!\right]$$
induces an isomorphism 
$$\{\sigma_b\in\Gamma(\Wedge^1)\mid\nabla_{(a}\sigma_{b)}=0\}
\cong
\{\Sigma_{\beta\gamma}\in\Gamma(\Wedge^2{\mathcal{T}})\mid
\nabla_a\Sigma_{\beta\gamma}=0\}$$
on the unit $n$-sphere.
\end{prop}
\begin{proof} The formula (\ref{killing_connection_on_sphere}) shows that the 
tractor connection on $\Wedge^2{\mathcal{T}}$ coincides with the Killing 
connection (\ref{the_killing_connection}) on the unit sphere. Hence, this 
isomorphism coincides with~(\ref{prolongation_isomorphism}). 
\end{proof}
As a minor variation on an already noted consequence, we have:
\begin{cor}
Let ${\mathbb{T}}$ denote the $(n+1)$-dimensional space of covariant constant
sections of~${\mathcal{T}}$.  Then we may identify the space of Killing fields
on $S^n$ with $\Wedge^2{\mathbb{T}}$.  In particular, the Killing fields on the
round sphere comprise an irreducible representation of
${\mathrm{SL}}(n+1,{\mathbb{R}})$ of dimension $n(n+1)/2$.
\end{cor}
To state the following theorem, we freely employ Young tableau as 
in~\cite{FH,OT}.
\begin{thm} Consider the induced Riemannian tractor connection on the bundle
$$\begin{picture}(100,30)(0,-10)
\put(0,0){\line(0,1){20}}
\put(10,0){\line(0,1){20}}
\put(20,0){\line(0,1){20}}
\put(50,0){\line(0,1){20}}
\put(60,0){\line(0,1){20}}
\put(0,0){\line(1,0){60}}
\put(0,10){\line(1,0){60}}
\put(0,20){\line(1,0){60}}
\put(36,4.5){\makebox(0,0){$\cdots$}}
\put(36,14.5){\makebox(0,0){$\cdots$}}
\put(74,10){\makebox(0,0){\large$({\mathcal{T}}).$}}
\put(30,-9.5){\makebox(0,0){$\underbrace{\hspace{58pt}}_k$}}
\end{picture}$$
It is a prolongation connection for the differential operator 
\eqref{killing_tensor_operator} on $S^n$, i.e.~there is a 
canonical surjective homomorphism of vector bundles
\begin{equation}\label{canonical_surjection}
\textstyle
\begin{picture}(85,12)(-3,7)
\put(0,0){\line(0,1){20}}
\put(10,0){\line(0,1){20}}
\put(20,0){\line(0,1){20}}
\put(50,0){\line(0,1){20}}
\put(60,0){\line(0,1){20}}
\put(0,0){\line(1,0){60}}
\put(0,10){\line(1,0){60}}
\put(0,20){\line(1,0){60}}
\put(36,4.5){\makebox(0,0){$\cdots$}}
\put(36,14.5){\makebox(0,0){$\cdots$}}
\put(72,10){\makebox(0,0){\large$({\mathcal{T}})$}}
\end{picture}\to\bigodot^k\!\Wedge^1\end{equation}
inducing a (local) isomorphism 
$$\textstyle\{\Sigma\in\Gamma\Big(\begin{picture}(85,12)(-3,7)
\put(0,0){\line(0,1){20}}
\put(10,0){\line(0,1){20}}
\put(20,0){\line(0,1){20}}
\put(50,0){\line(0,1){20}}
\put(60,0){\line(0,1){20}}
\put(0,0){\line(1,0){60}}
\put(0,10){\line(1,0){60}}
\put(0,20){\line(1,0){60}}
\put(36,4.5){\makebox(0,0){$\cdots$}}
\put(36,14.5){\makebox(0,0){$\cdots$}}
\put(72,10){\makebox(0,0){\large$({\mathcal{T}})$}}
\end{picture}\Big)\mid\nabla_a\Sigma=0\}\cong
\{\sigma_{bc\cdots d}\in\Gamma(\bigodot^k\!\Wedge^1)\mid
\nabla_{(a}\sigma_{bc\cdots d)}=0\}.$$
\end{thm}
\begin{proof} The canonical surjection $\Wedge^2{\mathcal{T}}\to\Wedge^1$ is
evident in (\ref{killing_connection_on_sphere}) and gives rise
to~(\ref{canonical_surjection}).  More precisely, 
$$\begin{picture}(45,12)(-3,7)
\put(0,0){\line(0,1){20}}
\put(10,0){\line(0,1){20}}
\put(20,0){\line(0,1){20}}
\put(0,0){\line(1,0){20}}
\put(0,10){\line(1,0){20}}
\put(0,20){\line(1,0){20}}
\put(32,10){\makebox(0,0){\large$({\mathcal{T}})$}}
\end{picture}=\begin{picture}(48,12)(-3,7)
\put(0,5){\line(0,1){10}}
\put(10,5){\line(0,1){10}}
\put(20,5){\line(0,1){10}}
\put(0,5){\line(1,0){20}}
\put(0,15){\line(1,0){20}}
\put(32,10){\makebox(0,0){$(\Wedge^1)$}}
\end{picture}\oplus\begin{picture}(48,12)(-3,7)
\put(0,0){\line(0,1){20}}
\put(10,0){\line(0,1){20}}
\put(20,10){\line(0,1){10}}
\put(0,0){\line(1,0){10}}
\put(0,10){\line(1,0){20}}
\put(0,20){\line(1,0){20}}
\put(32,10){\makebox(0,0){$(\Wedge^1)$}}
\end{picture}\oplus\begin{picture}(45,12)(-3,7)
\put(0,0){\line(0,1){20}}
\put(10,0){\line(0,1){20}}
\put(20,0){\line(0,1){20}}
\put(0,0){\line(1,0){20}}
\put(0,10){\line(1,0){20}}
\put(0,20){\line(1,0){20}}
\put(32,10){\makebox(0,0){$(\Wedge^1)$}}
\end{picture}\,,$$
and
$$\begin{picture}(55,12)(-3,7)
\put(0,0){\line(0,1){20}}
\put(10,0){\line(0,1){20}}
\put(20,0){\line(0,1){20}}
\put(30,0){\line(0,1){20}}
\put(0,0){\line(1,0){30}}
\put(0,10){\line(1,0){30}}
\put(0,20){\line(1,0){30}}
\put(42,10){\makebox(0,0){\large$({\mathcal{T}})$}}
\end{picture}=\begin{picture}(58,12)(-3,7)
\put(0,5){\line(0,1){10}}
\put(10,5){\line(0,1){10}}
\put(20,5){\line(0,1){10}}
\put(30,5){\line(0,1){10}}
\put(0,5){\line(1,0){30}}
\put(0,15){\line(1,0){30}}
\put(42,10){\makebox(0,0){$(\Wedge^1)$}}
\end{picture}\oplus\begin{picture}(58,12)(-3,7)
\put(0,0){\line(0,1){20}}
\put(10,0){\line(0,1){20}}
\put(20,10){\line(0,1){10}}
\put(30,10){\line(0,1){10}}
\put(0,0){\line(1,0){10}}
\put(0,10){\line(1,0){30}}
\put(0,20){\line(1,0){30}}
\put(42,10){\makebox(0,0){$(\Wedge^1)$}}
\end{picture}\oplus\begin{picture}(58,12)(-3,7)
\put(0,0){\line(0,1){20}}
\put(10,0){\line(0,1){20}}
\put(20,0){\line(0,1){20}}
\put(30,10){\line(0,1){10}}
\put(0,0){\line(1,0){20}}
\put(0,10){\line(1,0){30}}
\put(0,20){\line(1,0){30}}
\put(42,10){\makebox(0,0){$(\Wedge^1)$}}
\end{picture}\oplus\begin{picture}(55,12)(-3,7)
\put(0,0){\line(0,1){20}}
\put(10,0){\line(0,1){20}}
\put(20,0){\line(0,1){20}}
\put(30,0){\line(0,1){20}}
\put(0,0){\line(1,0){30}}
\put(0,10){\line(1,0){30}}
\put(0,20){\line(1,0){30}}
\put(42,10){\makebox(0,0){$(\Wedge^1)$}}
\end{picture}\,,$$

\smallskip\noindent
and so on. It is also easy to check that the induced connection takes the form
$$\nabla_a\left[\!\begin{array}{c}\sigma_{bc\cdots d}\\ 
\mu_{bc\cdots de}\\
\vdots\end{array}\!\right]
\!=\!\left[\!\begin{array}{c}\nabla_a\sigma_{bc\cdots d}-\mu_{abc\cdots d}\\ 
\nabla_a\mu_{bc\cdots de}-\cdots\\ 
\vdots\end{array}\!\right]\enskip\mbox{where}\enskip
\left\{\begin{array}{l}\sigma_{bc\cdots d}=\sigma_{(bc\cdots d)}\\
\mu_{bc\cdots de}=\mu_{b(c\cdots de)}\enskip\mbox{and}\enskip
\mu_{(bc\cdots de)}=0\\ \quad\vdots\end{array}\right.$$
so that $\nabla_a\Sigma=0\Rightarrow\nabla_{(a}\sigma_{bc\cdots d)}=0$.  That
the remaining equations encoded in $\nabla_a\Sigma=0$ follow from this first
equation is a consequence of the `BGG machinery,' as explained in~\cite{BCEG}
and especially~\cite[Theorem~2.1]{BCEG}.  For more details see \cite{E} and
\cite{EGover}, where one should bear in mind that, on the unit sphere, the
connection (\ref{R_tractors}) coincides with the costandard projective tractor
connection
$$\nabla_a\left[\!\begin{array}{c}\sigma\\ \mu_b\end{array}\!\right]
\!\equiv\!\left[\!\begin{array}{c}\nabla_a\sigma-\mu_a\\ 
\nabla_a\mu_b+\Rho_{ab}\sigma\end{array}\!\right],$$
where $\Rho_{ab}\equiv\frac1{n-1}R_{ab}$.
\end{proof}

Nowadays, this projective differential geometry viewpoint is commonly adopted
once one realises that (\ref{killing_tensor_operator}) is projectively
invariant.  Indeed, the Killing operator is just the first in a locally exact 
sequence of linear differential operators
$$0\to\begin{picture}(110,14)
\put(5,2){\makebox(0,0){$\bullet$}}
\put(25,2){\makebox(0,0){$\bullet$}}
\put(45,2){\makebox(0,0){$\bullet$}}
\put(65,2){\makebox(0,0){$\bullet$}}
\put(5,2){\line(1,0){70}}
\put(85,2){\makebox(0,0){$\cdots$}}
\put(95,2){\line(1,0){10}}
\put(105,2){\makebox(0,0){$\bullet$}}
\put(5,9){\makebox(0,0){$\scriptstyle 0$}}
\put(25,9){\makebox(0,0){$\scriptstyle k$}}
\put(45,9){\makebox(0,0){$\scriptstyle 0$}}
\put(65,9){\makebox(0,0){$\scriptstyle 0$}}
\put(105,9){\makebox(0,0){$\scriptstyle 0$}}
\end{picture}\to
\begin{picture}(110,14)
\put(5,2){\makebox(0,0){$\times$}}
\put(25,2){\makebox(0,0){$\bullet$}}
\put(45,2){\makebox(0,0){$\bullet$}}
\put(65,2){\makebox(0,0){$\bullet$}}
\put(5,2){\line(1,0){70}}
\put(85,2){\makebox(0,0){$\cdots$}}
\put(95,2){\line(1,0){10}}
\put(105,2){\makebox(0,0){$\bullet$}}
\put(5,9){\makebox(0,0){$\scriptstyle 0$}}
\put(25,9){\makebox(0,0){$\scriptstyle k$}}
\put(45,9){\makebox(0,0){$\scriptstyle 0$}}
\put(65,9){\makebox(0,0){$\scriptstyle 0$}}
\put(105,9){\makebox(0,0){$\scriptstyle 0$}}
\end{picture}\xrightarrow{\,\nabla\,}
\begin{picture}(110,14)
\put(5,2){\makebox(0,0){$\times$}}
\put(25,2){\makebox(0,0){$\bullet$}}
\put(45,2){\makebox(0,0){$\bullet$}}
\put(65,2){\makebox(0,0){$\bullet$}}
\put(5,2){\line(1,0){70}}
\put(85,2){\makebox(0,0){$\cdots$}}
\put(95,2){\line(1,0){10}}
\put(105,2){\makebox(0,0){$\bullet$}}
\put(5,9){\makebox(0,0){$\scriptstyle -2$}}
\put(25,9){\makebox(0,0){$\scriptstyle k+1$}}
\put(45,9){\makebox(0,0){$\scriptstyle 0$}}
\put(65,9){\makebox(0,0){$\scriptstyle 0$}}
\put(105,9){\makebox(0,0){$\scriptstyle 0$}}
\end{picture}\to\cdots$$
known as the `BGG complex,' which, as presented in~\cite{EGover} on the unit
sphere, generalises to resolve any finite-dimensional irreducible
representation of~${\mathrm{SL}}(n+1,{\mathbb{R}})$.

The original BGG complex is due to Bernstein-Gelfand-Gelfand~\cite{BGG},
formulated in terms of induced modules for any complex semisimple Lie algebra
${\mathfrak{g}}$ with Borel subalgebra~${\mathfrak{b}}$.  This complex was
generalised by Lepowsky~\cite{L}, replacing ${\mathfrak{b}}$ by an arbitrary
parabolic subalgebra~${\mathfrak{p}}$, and interpreted geometrically in
\cite[Chapter~8]{Beastwood} as locally exact complexes of invariant
differential operators on the corresponding flag manifolds~$G/P$.  The first of
these BGG operators has an irreducible finite-dimensional representation of
${\mathfrak{g}}$ as its kernel and, by design, all such irreducibles arise in
this way.  {From} this point of view, the Killing operator on $S^n$ descends to
${\mathbb{RP}}_n$ and complexifies to a first BGG operator on ${\mathbb{CP}}_n$
as a flag manifold for ${\mathrm{SL}}(n+1,{\mathbb{C}})$.  Such `real forms' 
as ${\mathbb{RP}}_n\subset{\mathbb{CP}}_n$ 
are nowadays well-incorporated into the theory of `parabolic differential 
geometry'~\cite{parabook}.

That the Killing tensors comprise an {\em irreducible\/} representation of 
the special linear group has been noted by many authors, e.g.~\cite{MMS,SY}.  
The dimension formula~(\ref{DTT}) was found independently by 
Delong~\cite{D}, Takeuchi~\cite{Takeuchi}, and Thompson~\cite{Thompson}.

\section{K\"ahlerian tractors}\label{Kaehlerian_tractors}

Before discussing K\"ahler manifolds in general, let us recall our choice of
scaling for the Fubini-Study metric on complex projective space so that 
(\ref{cpn_curvature}) holds. 
As explained in~\cite{EG}, this normalisation is so that
$S^n\to{\mathbb{RP}}_n\hookrightarrow{\mathbb{CP}}_n$ is totally geodesic for
the unit sphere (where $R_{abcd}=g_{ac}g_{bd}-g_{bc}g_{ad}$).

On a K\"ahler manifold, the tractor connection (\ref{R_tractors}) can be
modified as follows, especially to take account of the symplectic
structure~\cite{ES}:
\begin{equation}\label{K_tractors}
{\mathcal{T}}\equiv
\begin{array}{c}\Wedge^0\\[-3pt] \oplus\\[-1pt] \Wedge^1\\[-3pt]
\oplus\\[-1pt] \Wedge^0\end{array}\enskip\mbox{with connection}\enskip
\nabla_b\left[\!\begin{array}{c}\sigma\\ \mu_c\\ \rho\end{array}\!\right]
\!\equiv\!\left[\!\begin{array}{c}\nabla_b\sigma-\mu_b\\ 
\nabla_b\mu_c+g_{bc}\sigma+J_{bc}\rho\\
\nabla_b\rho-J_b{}^c\mu_c
\end{array}\!\right],\end{equation}
where $J_{ab}$ is the K\"ahler form and $J_a{}^b=J_{ac}g^{bc}$ is the complex
structure.  We shall find that this connection has better properties than 
(\ref{R_tractors}) and will henceforth use it exclusively. In particular, 
there is no need for a new notation for the basic tractor 
bundle~${\mathcal{T}}$. 
\begin{prop}\label{K_curvature}
The curvature of the K\"ahlerian tractor connection is given by
$$(\nabla_a\nabla_b-\nabla_b\nabla_a)
\left[\!\begin{array}c\sigma\\ \mu_c\\ \rho\end{array}\!\right]
=\left[\!\begin{array}c2J_{ab}\rho\\ 
g_{bc}\mu_a-g_{ac}\mu_b
+J_{bc}J_a{}^d\mu_d-J_{ac}J_b{}^d\mu_d
-R_{ab}{}^d{}_c\mu_d\\
-2J_{ab}\sigma
\end{array}\!\right]\!.$$
\end{prop}
\begin{proof} 
We compute
$$\nabla_a\nabla_b\left[\!\begin{array}c\sigma\\ \mu_c\\ \rho
\end{array}\!\right]
=\left[\!\begin{array}c
\nabla_a(\nabla_b\sigma-\mu_b)-(\nabla_b\mu_a+g_{ba}\sigma+J_{ba}\rho)\\ 
\nabla_a(\nabla_b\mu_c+g_{bc}\sigma+J_{bc}\rho)+g_{ac}(\nabla_b\sigma-\mu_b)
+J_{ac}(\nabla_b\rho-J_b{}^d\mu_d)\\
\nabla_a(\nabla_b\rho-J_b{}^c\mu_c)
-J_a{}^c(\nabla_b\mu_c+g_{bc}\sigma+J_{bc}\rho)\end{array}\!\right]$$
so
$$(\nabla_a\nabla_b-\nabla_b\nabla_a)
\left[\!\begin{array}c\sigma\\ \mu_c\\ \rho\end{array}\!\right]
=\left[\!\begin{array}c2J_{ab}\rho\\ 
(\nabla_a\nabla_b-\nabla_b\nabla_a)\mu_c+g_{bc}\mu_a-g_{ac}\mu_b
+J_{bc}J_a{}^d\mu_d-J_{ac}J_b{}^d\mu_d\\
-2J_{ab}\sigma
\end{array}\!\right]\!,$$
which expands as above.
\end{proof}
\begin{cor}\label{K_curvature_on_cpn}
On\/ ${\mathbb{CP}}_n$, the curvature of the K\"ahlerian tractor
connection is given by
$$(\nabla_a\nabla_b-\nabla_b\nabla_a)
\left[\!\begin{array}{c}\sigma\\ \mu_c\\ \rho\end{array}\!\right]
=2J_{ab}\left[\!\begin{array}{c}
\rho\\ 
J_c{}^d\mu_d\\
-\sigma\end{array}\!\right].$$
\end{cor}
\begin{proof} Immediate from Proposition~\ref{K_curvature}
and~(\ref{cpn_curvature}).  \end{proof}
	
This corollary says that
$(\nabla_a\nabla_b-\nabla_b\nabla_a)\Sigma=2J_{ab}\Theta\Sigma$, for a certain
endomorphism $\Theta$ of ${\mathcal{T}}$ and, with the terminology
of~\cite{ES}, this connection is said to be {\em symplectically flat\/}.  In
general, the K\"ahlerian tractor connection has several other nice properties:
\begin{prop}\label{nice_properties}
The tractor connection \eqref{K_tractors} respects
\begin{itemize}
\item the metric
\begin{equation}\label{K_tractor_metric}
g^{\beta\gamma}\Sigma_\beta\tilde\Sigma_\gamma\equiv
\left(\left[\!\begin{array}c\sigma\\ \mu_b\\ \rho\end{array}\!\right],
\left[\!\begin{array}c\tilde\sigma\\ \tilde\mu_c\\ 
\tilde\rho\end{array}\!\right]\right)
\equiv\sigma\tilde\sigma+g^{bc}\mu_b\tilde\mu_c+\rho\tilde\rho\,,
\end{equation}
\item the non-degenerate skew form 
\begin{equation}\label{K_tractor_symplectic_form}
J^{\beta\gamma}\Sigma_\beta\tilde\Sigma_\gamma\equiv
\left\langle\left[\!\begin{array}c\sigma\\ \mu_b\\ \rho\end{array}\!\right],
\left[\!\begin{array}c\tilde\sigma\\ \tilde\mu_c\\ 
\tilde\rho\end{array}\!\right]\right\rangle\equiv
\sigma\tilde\rho+J^{bc}\mu_b\tilde\mu_c-\rho\tilde\sigma.\end{equation}
\end{itemize}
\end{prop}
\begin{proof} We calculate that, for example,
$$\left(\nabla_a\left[\!\begin{array}c\sigma\\ \mu_b\\ \rho
\end{array}\!\right],
\left[\!\begin{array}c\tilde\sigma\\ \tilde\mu_c\\ \tilde\rho 
\end{array}\!\right]\right)
+\left(\left[\!\begin{array}c\sigma\\ \mu_b\\ \rho\end{array}\!\right],
\nabla_a\left[\!\begin{array}c\tilde\sigma\\ \tilde\mu_c\\ \tilde\rho 
\end{array}\!\right]\right)
=\nabla_a(\sigma\tilde\sigma+g^{bc}\mu_b\tilde\mu_c+\rho\tilde\rho).$$
(Preservation of the skew form (\ref{K_tractor_symplectic_form}) is a 
special case of~\cite[Proposition~1]{ES}.)
\end{proof}
The endomorphism $\Theta$, implicit in Corollary~\ref{K_curvature_on_cpn} can 
now be written as 
$$\Sigma_\alpha\mapsto J_\alpha{}^\beta\Sigma_\beta,\quad\mbox{where}\enskip
J_{\alpha\beta}=J_\alpha{}^\gamma g_{\beta\gamma}$$
so this Corollary now says that
\begin{equation}\label{K_curvature_with_indices}
(\nabla_a\nabla_b-\nabla_b\nabla_a)\Sigma_\alpha
=2J_{ab}J_\alpha{}^\beta\Sigma_\beta,\end{equation}
whilst Proposition~\ref{nice_properties} implies that
\begin{equation}\label{all_gadgets_are_parallel}
\nabla_ag_{\alpha\beta}=0,\quad 
\nabla_aJ_{\alpha\beta}=0,\quad\mbox{and}\enskip
\nabla_aJ_\alpha{}^\beta=0.\end{equation}
We are now in a position to obtain some real information from the
K\"ahlerian tractor connection~(\ref{K_tractors}), albeit rather trivial in the
first instance:
\begin{prop}
The differential equation
$$\nabla_a\nabla_b\sigma+g_{ab}\sigma=0$$
has no local solutions on~${\mathbb{CP}}_n$. 
\end{prop}
\begin{proof} As a mild variation on Proposition~\ref{R-prolong}, firstly let
us observe that the connection (\ref{K_tractors}) is also a prolongation
connection for the operator $\sigma\mapsto\nabla_a\nabla_b\sigma+g_{ab}\sigma$.
At first, this may seem like an error because it is evident from the second 
line of (\ref{K_tractors}) that, for a covariant constant section of 
${\mathcal{T}}$, the component $\rho$ must vanish.  But this is, in fact, 
perfectly consistent because
$$\begin{array}{rcl}\nabla_b\mu_c+g_{bc}\sigma=0&\Rightarrow&
(\nabla_a\nabla_b-\nabla_b\nabla_a)\mu_c
+g_{bc}\nabla_a\sigma-g_{ac}\nabla_b\sigma=0\\
&\Rightarrow&{}-R_{ab}{}^d{}_c\mu_d
+g_{bc}\mu_a-g_{ac}\mu_b=0\\
&\Rightarrow&{}J_{ac}J_b{}^d\mu_d-J_{bc}J_a{}^d\mu_d+2J_{ab}J_c{}^d\mu_d=0\\
\end{array}$$
and contracting this last equation with $J^{ab}$ gives $J_c{}^d\mu_d=0$, which 
is what we find on the last line of~(\ref{K_tractors}).  So, now we are 
required to show that ${\mathcal{T}}$ has no covariant constant sections. If 
$\Sigma_\alpha$ were such a section, then 
$$0=(\nabla_a\nabla_b-\nabla_b\nabla_a)\Sigma_\alpha
=2J_{ab}J_\alpha{}^\beta\Sigma_\beta$$
and we conclude that $J_\alpha{}^\beta\Sigma_\beta=0$. But 
$J_\eta{}^\alpha J_\alpha{}^\beta=-\delta_\eta{}^\beta$ 
so~$\Sigma_\beta=0$, as required. 
\end{proof}
This may seem like a rather convoluted route to a relatively simple conclusion
but the corresponding output for the bundle
$$\Wedge_\perp^2{\mathcal{T}}
\equiv\{\Sigma_{\alpha\beta}\in\Wedge^2{\mathcal{T}}\mid 
J^{\alpha\beta}\Sigma_{\alpha\beta}=0\}$$ 
is slightly more significant (cf.~Proposition~\ref{killing_fields_prolonged}):
\begin{prop}\label{more_significant}
On ${\mathbb{CP}}_n$ with its Fubini-Study metric, the induced tractor connection on
$\Wedge_\perp^2{\mathcal{T}}=\Wedge^1\oplus\Wedge^2\oplus\Wedge^1$ is a
prolongation connection for Killing fields.
\end{prop}
\begin{proof} As noted in~\cite[Equation~(31)]{EG}, this connection looks like
this:
\begin{equation}\label{looks_like_this}
\nabla_a\left[\!\begin{array}{c}\sigma_b\\ \mu_{bc}\\
\rho_b\end{array}\!\right]
=\left[\!\begin{array}{c}\nabla_a\sigma_b-\mu_{ab}\\ 
\nabla_a\mu_{bc}+g_{ab}\sigma_c-g_{ac}\sigma_b+J_{bc}J_a{}^d\sigma_d
+J_{ab}\rho_c-J_{ac}\rho_b-J_{bc}\rho_a\\
\nabla_a\rho_b+J_a{}^c\mu_{bc}
\end{array}\!\right],\end{equation}
and it is evident from its first line that covariant constancy 
forces $\nabla_{(a}\sigma_{b)}=0$.  Therefore, we are required to show that the
remaining lines contribute no additional constraints on~$\sigma_b$.  It is
immediate from $\nabla_a\sigma_b=\mu_{ab}$, that the $2$-form $\mu_{ab}$ is
closed and we already used this in (\ref{closure}) determine 
$\nabla_a\mu_{bc}$ and conclude that the system has
closed. We have arrived at a perfectly fine prolongation
connection~(\ref{the_killing_connection}), as noted in the proof of
Proposition~\ref{prolongation}.  For the Fubini-Study metric,
however, we may use (\ref{cpn_curvature}) to expand (\ref{closure}) as
$$\nabla_a\mu_{bc}+g_{ab}\sigma_c-g_{ac}\sigma_b+J_{bc}J_a{}^d\sigma_d=
J_{ab}J_c{}^d\sigma_d-J_{ac}J_b{}^d\sigma_d-J_{bc}J_a{}^d\sigma_d$$
and `inadvertently' introduce $\rho_b\equiv-J_b{}^d\sigma_d$, obtaining the
second line of (\ref{looks_like_this}) rather than immediate closure.  But now
$\nabla_a\rho_b+J_a{}^c\mu_{bc}=-\nabla_a(J_b{}^c\sigma_c)+J_a{}^c\mu_{bc}
=2J_{[a}{}^c\mu_{b]c}$ and we have already observed (\ref{mu_is_type_one_one})
that $J_{[a}{}^c\mu_{b]c}=0$ as a consequence of the curvature of
(\ref{the_killing_connection}) on~${\mathbb{CP}}_n$.  The last line of
(\ref{looks_like_this}) therefore contributes nothing new.
\end{proof}
\begin{cor}
The space of Killing fields on~${\mathbb{CP}}_n$ with its Fubini-Study metric
may be identified with the covariant constant sections of the parallel flat
subbundle
\begin{equation}\label{flat_and_parallel}\{\Sigma_{\alpha\beta}\mid 
J_\alpha{}^\gamma\Sigma_{\gamma\beta}+J_\beta{}^\gamma\Sigma_{\alpha\gamma}=0
\}\subset\Wedge_\perp^2{\mathcal{T}}.\end{equation}
\end{cor}
\begin{proof} According to Proposition~\ref{more_significant}, we are required
to identify the covariant constant sections of~$\Wedge_\perp^2{\mathcal{T}}$.
Well, if $\nabla_a\Sigma_{\alpha\beta}=0$, then
$(\nabla_a\nabla_b-\nabla_b\nabla_a)\Sigma_{\alpha\beta}=0$.  But
(\ref{K_curvature_with_indices}) gives
$$(\nabla_a\nabla_b-\nabla_b\nabla_a)\Sigma_{\alpha\beta}=
2J_{ab}(J_\alpha{}^\gamma\Sigma_{\gamma\beta}
+J_\beta{}^\gamma\Sigma_{\beta\gamma})$$
so $\Sigma_{\alpha\beta}$ is a section of the
subbundle~(\ref{flat_and_parallel}).  This subbundle is parallel because
$\nabla_aJ_\alpha{}^\beta=0$ in accordance
with~(\ref{all_gadgets_are_parallel}).
\end{proof}
This corollary coincides with Theorem~\ref{killing_tensors_on_cpn} when~$k=1$.
One subtlety to bear in mind, however, is that the space ${\mathbb{T}}$
appearing in Theorem~\ref{killing_tensors_on_cpn} is not canonical but rather
should be realised as the fibre of the K\"ahlerian tractor bundle
${\mathcal{T}}$ at a chosen basepoint in~${\mathbb{CP}}_n$.  Since the
connection on this bundle is only symplectically flat rather than flat, this
realisation very much depends on this choice.

In any case, to prove Theorem~\ref{killing_tensors_on_cpn} it remains to 
establish the following.
\begin{prop} On ${\mathbb{CP}}_n$ with its Fubini-Study metric, the induced 
K\"ahlerian tractor connection on 
$$\begin{picture}(92,24)(0,-11)
\put(0,0){\line(0,1){20}}
\put(10,0){\line(0,1){20}}
\put(20,0){\line(0,1){20}}
\put(50,0){\line(0,1){20}}
\put(60,0){\line(0,1){20}}
\put(0,0){\line(1,0){60}}
\put(0,10){\line(1,0){60}}
\put(0,20){\line(1,0){60}}
\put(36,4.5){\makebox(0,0){$\cdots$}}
\put(36,14.5){\makebox(0,0){$\cdots$}}
\put(65,2){\makebox(0,0){$\perp$}}
\put(77,10){\makebox(0,0){\large$({\mathcal{T}})$}}
\put(30,-9.5){\makebox(0,0){$\underbrace{\hspace{58pt}}_k$}}
\end{picture}\raisebox{18pt}{$=\bigodot^k\!\Wedge^1\oplus\cdots$}$$
is a prolongation connection for the rank~$k$ Killing tensors.
\end{prop}
\begin{proof} This is implicit in~\cite[\S5]{EG} and an explicit consequence
of~\cite[Theorem~3]{ES} by setting
$$\begin{picture}(130,14)(0,-2)
\put(5,2){\makebox(0,0){$\bullet$}}
\put(25,2){\makebox(0,0){$\bullet$}}
\put(45,2){\makebox(0,0){$\bullet$}}
\put(65,2){\makebox(0,0){$\bullet$}}
\put(85,2){\makebox(0,0){$\cdots$}}
\put(105,2){\makebox(0,0){$\bullet$}}
\put(125,2){\makebox(0,0){$\bullet$}}
\put(5,2){\line(1,0){70}}
\put(95,2){\line(1,0){10}}
\put(105,1){\line(1,0){20}}
\put(105,3){\line(1,0){20}}
\put(115,2){\makebox(0,0){$\langle$}}
\put(5,10){\makebox(0,0){$\scriptstyle a$}}
\put(25,10){\makebox(0,0){$\scriptstyle b$}}
\put(45,10){\makebox(0,0){$\scriptstyle c$}}
\put(65,10){\makebox(0,0){$\scriptstyle d$}}
\put(105,10){\makebox(0,0){$\scriptstyle e$}}
\put(125,10){\makebox(0,0){$\scriptstyle f$}}
\end{picture}
=\begin{picture}(130,14)(0,-2)
\put(5,2){\makebox(0,0){$\bullet$}}
\put(25,2){\makebox(0,0){$\bullet$}}
\put(45,2){\makebox(0,0){$\bullet$}}
\put(65,2){\makebox(0,0){$\bullet$}}
\put(85,2){\makebox(0,0){$\cdots$}}
\put(105,2){\makebox(0,0){$\bullet$}}
\put(125,2){\makebox(0,0){$\bullet$}}
\put(5,2){\line(1,0){70}}
\put(95,2){\line(1,0){10}}
\put(105,1){\line(1,0){20}}
\put(105,3){\line(1,0){20}}
\put(115,2){\makebox(0,0){$\langle$}}
\put(5,10){\makebox(0,0){$\scriptstyle 0$}}
\put(25,10){\makebox(0,0){$\scriptstyle k$}}
\put(45,10){\makebox(0,0){$\scriptstyle 0$}}
\put(65,10){\makebox(0,0){$\scriptstyle 0$}}
\put(105,10){\makebox(0,0){$\scriptstyle 0$}}
\put(125,10){\makebox(0,0){$\scriptstyle 0$}}
\end{picture}.$$
In short, the BGG machinery for $S^n$ entails the Lie algebra cohomology
$H^r({\mathfrak{g}}_{-1},{\mathbb{V}})$ where ${\mathfrak{g}}_{-1}$ is Abelian
whilst the BGG machinery for ${\mathbb{CP}}_n$ entails the Lie algebra 
cohomology $H^r({\mathfrak{h}},{\mathbb{V}})$ where ${\mathfrak{h}}$ is the 
$(2n+1)$-dimensional Heisenberg algebra.
\end{proof}

\section{Dimension formul{\ae}}
In this section, we derive formul{\ae} for the dimensions of the spaces in
Theorem~\ref{killing_tensors_on_cpn}.  This is a question of representation
theory of the appropriate groups, namely ${\mathrm{SL}}(2n+2,{\mathbb{R}})$ and
its subgroups 
$${\mathrm{SU}}(n+1)
={\mathrm{SL}}(n+1,{\mathbb{C}})\cap{\mathrm{Sp}}(2n+2,{\mathbb{R}})$$
acting on the irreducible representation 
$$\begin{picture}(85,23)(0,-8)
\put(0,0){\line(0,1){20}}
\put(10,0){\line(0,1){20}}
\put(20,0){\line(0,1){20}}
\put(50,0){\line(0,1){20}}
\put(60,0){\line(0,1){20}}
\put(0,0){\line(1,0){60}}
\put(0,10){\line(1,0){60}}
\put(0,20){\line(1,0){60}}
\put(36,4.5){\makebox(0,0){$\cdots$}}
\put(36,14.5){\makebox(0,0){$\cdots$}}
\put(72,10){\makebox(0,0){\large$({\mathbb{T}})$}}
\put(30,-9.5){\makebox(0,0){$\underbrace{\hspace{58pt}}_k$}}
\end{picture}$$
in the first instance and then, more conveniently, on
\begin{equation}\label{convenient}\begin{picture}(112,8)(0,5)
\put(0,0){\line(0,1){20}}
\put(10,0){\line(0,1){20}}
\put(20,0){\line(0,1){20}}
\put(50,0){\line(0,1){20}}
\put(60,0){\line(0,1){20}}
\put(0,0){\line(1,0){60}}
\put(0,10){\line(1,0){60}}
\put(0,20){\line(1,0){60}}
\put(36,4.5){\makebox(0,0){$\cdots$}}
\put(36,14.5){\makebox(0,0){$\cdots$}}
\put(86,10){\makebox(0,0){\large$({\mathbb{C}}\otimes{\mathbb{T}})$}}
\put(30,-9.5){\makebox(0,0){$\underbrace{\hspace{58pt}}_k$}}
\end{picture}=\underbrace{\begin{picture}(170,14)
\put(5,2){\makebox(0,0){$\bullet$}}
\put(25,2){\makebox(0,0){$\bullet$}}
\put(45,2){\makebox(0,0){$\bullet$}}
\put(65,2){\makebox(0,0){$\bullet$}}
\put(85,2){\makebox(0,0){$\bullet$}}
\put(105,2){\makebox(0,0){$\bullet$}}
\put(125,2){\makebox(0,0){$\bullet$}}
\put(145,2){\makebox(0,0){$\bullet$}}
\put(165,2){\makebox(0,0){$\bullet$}}
\put(5,2){\line(1,0){160}}
\put(5,9){\makebox(0,0){$\scriptstyle 0$}}
\put(25,9){\makebox(0,0){$\scriptstyle k$}}
\put(45,9){\makebox(0,0){$\scriptstyle 0$}}
\put(65,9){\makebox(0,0){$\scriptstyle 0$}}
\put(85,9){\makebox(0,0){$\scriptstyle 0$}}
\put(105,9){\makebox(0,0){$\scriptstyle 0$}}
\put(125,9){\makebox(0,0){$\scriptstyle 0$}}
\put(145,9){\makebox(0,0){$\scriptstyle 0$}}
\put(165,9){\makebox(0,0){$\scriptstyle 0$}}
\end{picture}}_{2n+1\mbox{ \scriptsize nodes}}\end{equation}

\smallskip\noindent without changing anything provided that we now compute
complex dimensions instead of real.  The advantage of complexifying a vector
space that is already complex, in this case ${\mathbb{T}}$ equipped with the
action of $J_\alpha{}^\beta$, is well-known, namely that tensors now decompose
into various `types' (see, e.g.~\cite[\S1.1]{CEMN}).  Moreover, with the
derivation action~(\ref{derivation_action}), the constraint
$(J\Sigma)_{\alpha\beta\cdots\eta}=0$ on the $2k$-tensor
$\Sigma_{\alpha\beta\cdots\eta}$ is simply that it be of type $(k,k)$.  For
convenience, let us now write Dynkin diagrams with the appropriate number of
nodes for~$n=4$ (as we already did with~(\ref{convenient})).  With these
preliminaries in mind, for the real dimension of~(\ref{key_vector_space}), we
are, equivalently, asking for
$$\dim_{\mathbb{C}}(\underbrace{\begin{picture}(90,14)
\put(5,2){\makebox(0,0){$\bullet$}}
\put(25,2){\makebox(0,0){$\bullet$}}
\put(45,2){\makebox(0,0){$\bullet$}}
\put(65,2){\makebox(0,0){$\bullet$}}
\put(85,2){\makebox(0,0){$\bullet$}}
\put(5,2){\line(1,0){60}}
\put(65,1){\line(1,0){20}}
\put(65,3){\line(1,0){20}}
\put(75,2){\makebox(0,0){$\langle$}}
\put(5,10){\makebox(0,0){$\scriptstyle 0$}}
\put(25,10){\makebox(0,0){$\scriptstyle k$}}
\put(45,10){\makebox(0,0){$\scriptstyle 0$}}
\put(65,10){\makebox(0,0){$\scriptstyle 0$}}
\put(85,10){\makebox(0,0){$\scriptstyle 0$}}
\end{picture}}_{n+1\mbox{ \scriptsize nodes}})^{k,k}.$$
To compute this, let us keep the `trace-free with respect to 
$J^{\alpha\beta}$' requirement on hold for the moment and start by branching 
from ${\mathrm{SL}}(2n+2,{\mathbb{R}})$ to ${\mathrm{SL}}(n+1,{\mathbb{C}})$, 
to obtain
$$(\begin{picture}(170,14)
\put(5,2){\makebox(0,0){$\bullet$}}
\put(25,2){\makebox(0,0){$\bullet$}}
\put(45,2){\makebox(0,0){$\bullet$}}
\put(65,2){\makebox(0,0){$\bullet$}}
\put(85,2){\makebox(0,0){$\bullet$}}
\put(105,2){\makebox(0,0){$\bullet$}}
\put(125,2){\makebox(0,0){$\bullet$}}
\put(145,2){\makebox(0,0){$\bullet$}}
\put(165,2){\makebox(0,0){$\bullet$}}
\put(5,2){\line(1,0){160}}
\put(5,9){\makebox(0,0){$\scriptstyle 0$}}
\put(25,9){\makebox(0,0){$\scriptstyle k$}}
\put(45,9){\makebox(0,0){$\scriptstyle 0$}}
\put(65,9){\makebox(0,0){$\scriptstyle 0$}}
\put(85,9){\makebox(0,0){$\scriptstyle 0$}}
\put(105,9){\makebox(0,0){$\scriptstyle 0$}}
\put(125,9){\makebox(0,0){$\scriptstyle 0$}}
\put(145,9){\makebox(0,0){$\scriptstyle 0$}}
\put(165,9){\makebox(0,0){$\scriptstyle 0$}}
\end{picture})^{k,k}
=\!\!\bigoplus_{p+2q=k}
\!\!\begin{picture}(70,14)
\put(5,2){\makebox(0,0){$\bullet$}}
\put(25,2){\makebox(0,0){$\bullet$}}
\put(45,2){\makebox(0,0){$\bullet$}}
\put(65,2){\makebox(0,0){$\bullet$}}
\put(5,2){\line(1,0){60}}
\put(5,9){\makebox(0,0){$\scriptstyle p$}}
\put(25,9){\makebox(0,0){$\scriptstyle q$}}
\put(45,9){\makebox(0,0){$\scriptstyle 0$}}
\put(65,9){\makebox(0,0){$\scriptstyle 0$}}
\end{picture}(\Wedge^{1,0})
\otimes\begin{picture}(70,14)
\put(5,2){\makebox(0,0){$\bullet$}}
\put(25,2){\makebox(0,0){$\bullet$}}
\put(45,2){\makebox(0,0){$\bullet$}}
\put(65,2){\makebox(0,0){$\bullet$}}
\put(5,2){\line(1,0){60}}
\put(5,9){\makebox(0,0){$\scriptstyle p$}}
\put(25,9){\makebox(0,0){$\scriptstyle q$}}
\put(45,9){\makebox(0,0){$\scriptstyle 0$}}
\put(65,9){\makebox(0,0){$\scriptstyle 0$}}
\end{picture}(\Wedge^{0,1})$$
and it follows that
$$\dim_{\mathbb{C}}(\begin{picture}(170,14)
\put(5,2){\makebox(0,0){$\bullet$}}
\put(25,2){\makebox(0,0){$\bullet$}}
\put(45,2){\makebox(0,0){$\bullet$}}
\put(65,2){\makebox(0,0){$\bullet$}}
\put(85,2){\makebox(0,0){$\bullet$}}
\put(105,2){\makebox(0,0){$\bullet$}}
\put(125,2){\makebox(0,0){$\bullet$}}
\put(145,2){\makebox(0,0){$\bullet$}}
\put(165,2){\makebox(0,0){$\bullet$}}
\put(5,2){\line(1,0){160}}
\put(5,9){\makebox(0,0){$\scriptstyle 0$}}
\put(25,9){\makebox(0,0){$\scriptstyle k$}}
\put(45,9){\makebox(0,0){$\scriptstyle 0$}}
\put(65,9){\makebox(0,0){$\scriptstyle 0$}}
\put(85,9){\makebox(0,0){$\scriptstyle 0$}}
\put(105,9){\makebox(0,0){$\scriptstyle 0$}}
\put(125,9){\makebox(0,0){$\scriptstyle 0$}}
\put(145,9){\makebox(0,0){$\scriptstyle 0$}}
\put(165,9){\makebox(0,0){$\scriptstyle 0$}}
\end{picture})^{k,k}=\!\!
\sum_{p+2q=k}\left(\frac{(n+q-1)!(n+p+q)!(p+1)}{q!(p+q+1)!(n-1)!n!}\right)^2.$$

We may impose the `trace-free with respect to $J^{\alpha\beta}$' constraint
separately.  Specifically, there are evident short exact sequences
$$\begin{array}{ccccccc}0&\to&\begin{picture}(90,14)
\put(5,2){\makebox(0,0){$\bullet$}}
\put(25,2){\makebox(0,0){$\bullet$}}
\put(45,2){\makebox(0,0){$\bullet$}}
\put(65,2){\makebox(0,0){$\bullet$}}
\put(85,2){\makebox(0,0){$\bullet$}}
\put(5,2){\line(1,0){60}}
\put(65,1){\line(1,0){20}}
\put(65,3){\line(1,0){20}}
\put(75,2){\makebox(0,0){$\langle$}}
\put(5,10){\makebox(0,0){$\scriptstyle 0$}}
\put(25,10){\makebox(0,0){$\scriptstyle k$}}
\put(45,10){\makebox(0,0){$\scriptstyle 0$}}
\put(65,10){\makebox(0,0){$\scriptstyle 0$}}
\put(85,10){\makebox(0,0){$\scriptstyle 0$}}
\end{picture}
&\to&\begin{picture}(170,14)
\put(5,2){\makebox(0,0){$\bullet$}}
\put(25,2){\makebox(0,0){$\bullet$}}
\put(45,2){\makebox(0,0){$\bullet$}}
\put(65,2){\makebox(0,0){$\bullet$}}
\put(85,2){\makebox(0,0){$\bullet$}}
\put(105,2){\makebox(0,0){$\bullet$}}
\put(125,2){\makebox(0,0){$\bullet$}}
\put(145,2){\makebox(0,0){$\bullet$}}
\put(165,2){\makebox(0,0){$\bullet$}}
\put(5,2){\line(1,0){160}}
\put(5,9){\makebox(0,0){$\scriptstyle 0$}}
\put(25,9){\makebox(0,0){$\scriptstyle k$}}
\put(45,9){\makebox(0,0){$\scriptstyle 0$}}
\put(65,9){\makebox(0,0){$\scriptstyle 0$}}
\put(85,9){\makebox(0,0){$\scriptstyle 0$}}
\put(105,9){\makebox(0,0){$\scriptstyle 0$}}
\put(125,9){\makebox(0,0){$\scriptstyle 0$}}
\put(145,9){\makebox(0,0){$\scriptstyle 0$}}
\put(165,9){\makebox(0,0){$\scriptstyle 0$}}
\end{picture}\\
&&&&J\downarrow\phantom{J}\\
&&&&\begin{picture}(170,14)
\put(5,2){\makebox(0,0){$\bullet$}}
\put(25,2){\makebox(0,0){$\bullet$}}
\put(45,2){\makebox(0,0){$\bullet$}}
\put(65,2){\makebox(0,0){$\bullet$}}
\put(85,2){\makebox(0,0){$\bullet$}}
\put(105,2){\makebox(0,0){$\bullet$}}
\put(125,2){\makebox(0,0){$\bullet$}}
\put(145,2){\makebox(0,0){$\bullet$}}
\put(165,2){\makebox(0,0){$\bullet$}}
\put(5,2){\line(1,0){160}}
\put(5,9){\makebox(0,0){$\scriptstyle 0$}}
\put(25,9){\makebox(0,0){$\scriptstyle k-1$}}
\put(45,9){\makebox(0,0){$\scriptstyle 0$}}
\put(65,9){\makebox(0,0){$\scriptstyle 0$}}
\put(85,9){\makebox(0,0){$\scriptstyle 0$}}
\put(105,9){\makebox(0,0){$\scriptstyle 0$}}
\put(125,9){\makebox(0,0){$\scriptstyle 0$}}
\put(145,9){\makebox(0,0){$\scriptstyle 0$}}
\put(165,9){\makebox(0,0){$\scriptstyle 0$}}
\end{picture}&\to&0\end{array}$$
and, consequently, the short exact sequences
$$\begin{array}{ccccccc}0&\to&(\begin{picture}(90,14)
\put(5,2){\makebox(0,0){$\bullet$}}
\put(25,2){\makebox(0,0){$\bullet$}}
\put(45,2){\makebox(0,0){$\bullet$}}
\put(65,2){\makebox(0,0){$\bullet$}}
\put(85,2){\makebox(0,0){$\bullet$}}
\put(5,2){\line(1,0){60}}
\put(65,1){\line(1,0){20}}
\put(65,3){\line(1,0){20}}
\put(75,2){\makebox(0,0){$\langle$}}
\put(5,10){\makebox(0,0){$\scriptstyle 0$}}
\put(25,10){\makebox(0,0){$\scriptstyle k$}}
\put(45,10){\makebox(0,0){$\scriptstyle 0$}}
\put(65,10){\makebox(0,0){$\scriptstyle 0$}}
\put(85,10){\makebox(0,0){$\scriptstyle 0$}}
\end{picture})^{k,k}
&\!\to\!&(\begin{picture}(170,14)
\put(5,2){\makebox(0,0){$\bullet$}}
\put(25,2){\makebox(0,0){$\bullet$}}
\put(45,2){\makebox(0,0){$\bullet$}}
\put(65,2){\makebox(0,0){$\bullet$}}
\put(85,2){\makebox(0,0){$\bullet$}}
\put(105,2){\makebox(0,0){$\bullet$}}
\put(125,2){\makebox(0,0){$\bullet$}}
\put(145,2){\makebox(0,0){$\bullet$}}
\put(165,2){\makebox(0,0){$\bullet$}}
\put(5,2){\line(1,0){160}}
\put(5,9){\makebox(0,0){$\scriptstyle 0$}}
\put(25,9){\makebox(0,0){$\scriptstyle k$}}
\put(45,9){\makebox(0,0){$\scriptstyle 0$}}
\put(65,9){\makebox(0,0){$\scriptstyle 0$}}
\put(85,9){\makebox(0,0){$\scriptstyle 0$}}
\put(105,9){\makebox(0,0){$\scriptstyle 0$}}
\put(125,9){\makebox(0,0){$\scriptstyle 0$}}
\put(145,9){\makebox(0,0){$\scriptstyle 0$}}
\put(165,9){\makebox(0,0){$\scriptstyle 0$}}
\end{picture})^{k,k}\\
&&&&J\downarrow\phantom{J}\\
&&&&(\begin{picture}(170,14)
\put(5,2){\makebox(0,0){$\bullet$}}
\put(25,2){\makebox(0,0){$\bullet$}}
\put(45,2){\makebox(0,0){$\bullet$}}
\put(65,2){\makebox(0,0){$\bullet$}}
\put(85,2){\makebox(0,0){$\bullet$}}
\put(105,2){\makebox(0,0){$\bullet$}}
\put(125,2){\makebox(0,0){$\bullet$}}
\put(145,2){\makebox(0,0){$\bullet$}}
\put(165,2){\makebox(0,0){$\bullet$}}
\put(5,2){\line(1,0){160}}
\put(5,9){\makebox(0,0){$\scriptstyle 0$}}
\put(25,9){\makebox(0,0){$\scriptstyle k-1$}}
\put(45,9){\makebox(0,0){$\scriptstyle 0$}}
\put(65,9){\makebox(0,0){$\scriptstyle 0$}}
\put(85,9){\makebox(0,0){$\scriptstyle 0$}}
\put(105,9){\makebox(0,0){$\scriptstyle 0$}}
\put(125,9){\makebox(0,0){$\scriptstyle 0$}}
\put(145,9){\makebox(0,0){$\scriptstyle 0$}}
\put(165,9){\makebox(0,0){$\scriptstyle 0$}}
\end{picture})^{k-1,k-1}&\to&0\end{array}$$
from which the final dimension formula can be read off, the result being as 
follows.
\begin{thm} The dimension of the vector space \eqref{key_vector_space} is
$$\sum_{p+2q=k}
\!\!\left(\frac{(n+q-1)!(n+p+q)!(p+1)}{q!(p+q+1)!(n-1)!n!}\right)^2
\enskip
-\!\!\sum_{p+2q=k-1}
\!\!\left(\frac{(n+q-1)!(n+p+q)!(p+1)}{q!(p+q+1)!(n-1)!n!}\right)^2.$$
\end{thm}
\begin{cor} The space of rank $k$ Killing tensors on ${\mathbb{CP}}_n$ is 
finite-dimensional and
\begin{center}\begin{tabular}{rcl}
$k=1$&$\Rightarrow$& dimension is $n(n+2)$\\
$k=2$&$\Rightarrow$& dimension is $n(n+1)^2(n+2)/2$\\
$k=3$&$\Rightarrow$& dimension is
$n(n+1)^2(5n^3+25n^2+35n+24)/36$\\
$k=4$&$\Rightarrow$& dimension is
$n(n+1)^2(n+2)^2(7n^3+38n^2+39n+36)/288$
\end{tabular}\end{center}
\end{cor}
\begin{proof}
These are obtained by substituting these particular values of $k$ into the 
general formula (with aid from a computer). For higher $k$, one runs into 
serious expressions. 
\end{proof}	
\noindent Here are some numerical values for the dimension 
of~(\ref{key_vector_space}) for small $k$ and $n$.
$$\begin{array}{cccccccc}&{\mathbb{CP}}_1&{\mathbb{CP}}_2&{\mathbb{CP}}_3
&{\mathbb{CP}}_4&{\mathbb{CP}}_5&{\mathbb{CP}}_6&{\mathbb{CP}}_7\\
\mbox{ rank $1$}&3&8&15&24&35&48&63\\
\mbox{ rank $2$}&6&36&120&300&630&1176&2016\\
\mbox{ rank $3$}&10&119&664&2500&7370&18375&40544\\
\mbox{ rank $4$}&15&322&2850&15600&62965&205800&576072\\
\mbox{ rank $5$}&21&756&10142&78252&422919&1782032&6246072\\
\end{array}$$
Another way to find these numbers is to branch the irreducible representation
$$\underbrace
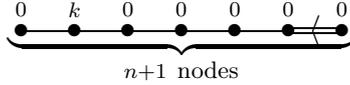
{\begin{picture}(130,14)
\put(5,2){\makebox(0,0){$\bullet$}}
\put(25,2){\makebox(0,0){$\bullet$}}
\put(45,2){\makebox(0,0){$\bullet$}}
\put(65,2){\makebox(0,0){$\bullet$}}
\put(85,2){\makebox(0,0){$\bullet$}}
\put(105,2){\makebox(0,0){$\bullet$}}
\put(125,2){\makebox(0,0){$\bullet$}}
\put(5,2){\line(1,0){100}}
\put(105,1){\line(1,0){20}}
\put(105,3){\line(1,0){20}}
\put(115,2){\makebox(0,0){$\langle$}}
\put(5,10){\makebox(0,0){$\scriptstyle 0$}}
\put(25,10){\makebox(0,0){$\scriptstyle k$}}
\put(45,10){\makebox(0,0){$\scriptstyle 0$}}
\put(65,10){\makebox(0,0){$\scriptstyle 0$}}
\put(85,10){\makebox(0,0){$\scriptstyle 0$}}
\put(105,10){\makebox(0,0){$\scriptstyle 0$}}
\put(125,10){\makebox(0,0){$\scriptstyle 0$}}
\end{picture}}_{n+1\mbox{ \scriptsize nodes}}$$
from ${\mathrm{Sp}}(2n+2,{\mathbb{R}})$ to ${\mathrm{SU}}(n+1)$, pick out those 
irreducibles on which $J_\alpha{}^\beta$ acts trivially, and add the 
dimensions of these parts.  On the level of Lie algebras, one is branching for 
the Levi subalgebra
$$A_n\times A_1\hookrightarrow C_{n+1}\quad\mbox{corresponding to}\quad
\begin{picture}(130,14)
\put(5,2){\makebox(0,0){$\bullet$}}
\put(25,2){\makebox(0,0){$\bullet$}}
\put(45,2){\makebox(0,0){$\bullet$}}
\put(65,2){\makebox(0,0){$\bullet$}}
\put(85,2){\makebox(0,0){$\bullet$}}
\put(105,2){\makebox(0,0){$\bullet$}}
\put(125,2){\makebox(0,0){$\times$}}
\put(5,2){\line(1,0){100}}
\put(105,1){\line(1,0){19}}
\put(105,3){\line(1,0){19}}
\put(115,2){\makebox(0,0){$\langle$}}\end{picture}$$
and extracting those irreducibles on which $A_1$ acts trivially. A suitable 
computer program for doing this can be found here~\cite[Display~(2.6)]{EW}.

For rank $2$ Killing fields on ${\mathbb{CP}}_n$, we find that
$$\begin{array}{rcl}{\mathrm{Sp}}(2n+2,{\mathbb{R}})
&\supset&{\mathrm{SU}}(n+1)\hspace{68pt}\mbox{Dimensions}\\
(\underbrace{\begin{picture}(130,14)
\put(5,2){\makebox(0,0){$\bullet$}}
\put(25,2){\makebox(0,0){$\bullet$}}
\put(45,2){\makebox(0,0){$\bullet$}}
\put(65,2){\makebox(0,0){$\bullet$}}
\put(85,2){\makebox(0,0){$\bullet$}}
\put(105,2){\makebox(0,0){$\bullet$}}
\put(125,2){\makebox(0,0){$\bullet$}}
\put(5,2){\line(1,0){100}}
\put(105,1){\line(1,0){20}}
\put(105,3){\line(1,0){20}}
\put(115,2){\makebox(0,0){$\langle$}}
\put(5,10){\makebox(0,0){$\scriptstyle 0$}}
\put(25,10){\makebox(0,0){$\scriptstyle 2$}}
\put(45,10){\makebox(0,0){$\scriptstyle 0$}}
\put(65,10){\makebox(0,0){$\scriptstyle 0$}}
\put(85,10){\makebox(0,0){$\scriptstyle 0$}}
\put(105,10){\makebox(0,0){$\scriptstyle 0$}}
\put(125,10){\makebox(0,0){$\scriptstyle 0$}}
\end{picture}}_{n+1\mbox{ \scriptsize nodes}})^{2,2}
&=&\!\!\!\begin{array}[t]{cl}\begin{picture}(110,14)
\put(5,2){\makebox(0,0){$\bullet$}}
\put(25,2){\makebox(0,0){$\bullet$}}
\put(45,2){\makebox(0,0){$\bullet$}}
\put(65,2){\makebox(0,0){$\bullet$}}
\put(85,2){\makebox(0,0){$\bullet$}}
\put(105,2){\makebox(0,0){$\bullet$}}
\put(5,2){\line(1,0){100}}
\put(5,10){\makebox(0,0){$\scriptstyle 2$}}
\put(25,10){\makebox(0,0){$\scriptstyle 0$}}
\put(45,10){\makebox(0,0){$\scriptstyle 0$}}
\put(65,10){\makebox(0,0){$\scriptstyle 0$}}
\put(85,10){\makebox(0,0){$\scriptstyle 0$}}
\put(105,10){\makebox(0,0){$\scriptstyle 2$}}
\end{picture}&n(n+1)^2(n+4)/4\\[-5pt]
\oplus\\[-6pt]
\begin{picture}(110,14)
\put(5,2){\makebox(0,0){$\bullet$}}
\put(25,2){\makebox(0,0){$\bullet$}}
\put(45,2){\makebox(0,0){$\bullet$}}
\put(65,2){\makebox(0,0){$\bullet$}}
\put(85,2){\makebox(0,0){$\bullet$}}
\put(105,2){\makebox(0,0){$\bullet$}}
\put(5,2){\line(1,0){100}}
\put(5,10){\makebox(0,0){$\scriptstyle 0$}}
\put(25,10){\makebox(0,0){$\scriptstyle 1$}}
\put(45,10){\makebox(0,0){$\scriptstyle 0$}}
\put(65,10){\makebox(0,0){$\scriptstyle 0$}}
\put(85,10){\makebox(0,0){$\scriptstyle 1$}}
\put(105,10){\makebox(0,0){$\scriptstyle 0$}}
\end{picture}&(n-2)(n+1)^2(n+2)/4\\[-5pt]
\oplus\\[-6pt]
\begin{picture}(110,14)
\put(5,2){\makebox(0,0){$\bullet$}}
\put(25,2){\makebox(0,0){$\bullet$}}
\put(45,2){\makebox(0,0){$\bullet$}}
\put(65,2){\makebox(0,0){$\bullet$}}
\put(85,2){\makebox(0,0){$\bullet$}}
\put(105,2){\makebox(0,0){$\bullet$}}
\put(5,2){\line(1,0){100}}
\put(5,10){\makebox(0,0){$\scriptstyle 1$}}
\put(25,10){\makebox(0,0){$\scriptstyle 0$}}
\put(45,10){\makebox(0,0){$\scriptstyle 0$}}
\put(65,10){\makebox(0,0){$\scriptstyle 0$}}
\put(85,10){\makebox(0,0){$\scriptstyle 0$}}
\put(105,10){\makebox(0,0){$\scriptstyle 1$}}
\end{picture}&n(n+1)\\[-5pt]
\oplus\\[-6pt]
\begin{picture}(110,14)
\put(5,2){\makebox(0,0){$\bullet$}}
\put(25,2){\makebox(0,0){$\bullet$}}
\put(45,2){\makebox(0,0){$\bullet$}}
\put(65,2){\makebox(0,0){$\bullet$}}
\put(85,2){\makebox(0,0){$\bullet$}}
\put(105,2){\makebox(0,0){$\bullet$}}
\put(5,2){\line(1,0){100}}
\put(5,10){\makebox(0,0){$\scriptstyle 0$}}
\put(25,10){\makebox(0,0){$\scriptstyle 0$}}
\put(45,10){\makebox(0,0){$\scriptstyle 0$}}
\put(65,10){\makebox(0,0){$\scriptstyle 0$}}
\put(85,10){\makebox(0,0){$\scriptstyle 0$}}
\put(105,10){\makebox(0,0){$\scriptstyle 0$}}
\end{picture}&1\\ \hline\hline\\[-10pt]
\makebox[70pt][l]{Total dimension = $n(n+1)^2(n+2)/2$}
\end{array}\end{array}$$
This decomposition is given by Delong~\cite[\S(6.2)]{D} under the assumption
that these Killing tensors of rank $2$ are generated by the Killing fields (as
we shall see to be true in general for Killing tensors of arbitrary rank
in~\S\ref{generation}), and he observes that
$$\textstyle\bigodot^2{\mathfrak{su}}(n+1)\to
(\underbrace{\begin{picture}(90,14)
\put(5,2){\makebox(0,0){$\bullet$}}
\put(25,2){\makebox(0,0){$\bullet$}}
\put(45,2){\makebox(0,0){$\bullet$}}
\put(65,2){\makebox(0,0){$\bullet$}}
\put(85,2){\makebox(0,0){$\bullet$}}
\put(5,2){\line(1,0){60}}
\put(65,1){\line(1,0){20}}
\put(65,3){\line(1,0){20}}
\put(75,2){\makebox(0,0){$\langle$}}
\put(5,10){\makebox(0,0){$\scriptstyle 0$}}
\put(25,10){\makebox(0,0){$\scriptstyle 2$}}
\put(45,10){\makebox(0,0){$\scriptstyle 0$}}
\put(65,10){\makebox(0,0){$\scriptstyle 0$}}
\put(85,10){\makebox(0,0){$\scriptstyle 0$}}
\end{picture}}_{n+1\mbox{ \scriptsize nodes}})^{2,2}$$
is then, not only surjective, but actually an isomorphism.  Regarding rank
$k=3$ on ${\mathbb{CP}}_2$, he shows that
$$\textstyle\bigodot^2{\mathfrak{su}}(3)\to
(\begin{picture}(50,14)
\put(5,2){\makebox(0,0){$\bullet$}}
\put(25,2){\makebox(0,0){$\bullet$}}
\put(45,2){\makebox(0,0){$\bullet$}}
\put(5,2){\line(1,0){20}}
\put(25,1){\line(1,0){20}}
\put(25,3){\line(1,0){20}}
\put(35,2){\makebox(0,0){$\langle$}}
\put(5,10){\makebox(0,0){$\scriptstyle 0$}}
\put(25,10){\makebox(0,0){$\scriptstyle 3$}}
\put(45,10){\makebox(0,0){$\scriptstyle 0$}}
\end{picture})^{3,3}$$
has just a $1$-dimensional kernel so that 
$$\dim(\begin{picture}(50,14)
\put(5,2){\makebox(0,0){$\bullet$}}
\put(25,2){\makebox(0,0){$\bullet$}}
\put(45,2){\makebox(0,0){$\bullet$}}
\put(5,2){\line(1,0){20}}
\put(25,1){\line(1,0){20}}
\put(25,3){\line(1,0){20}}
\put(35,2){\makebox(0,0){$\langle$}}
\put(5,10){\makebox(0,0){$\scriptstyle 0$}}
\put(25,10){\makebox(0,0){$\scriptstyle 3$}}
\put(45,10){\makebox(0,0){$\scriptstyle 0$}}
\end{picture})^{3,3}=\frac{8\times 9\times 10}{6}-1=119$$
and finds, for Killing tensors of arbitrary rank on ${\mathbb{CP}}_2$, that
\begin{equation}\label{delong_poincare_series}
\sum_{k\geq 0}\dim(\begin{picture}(50,14)
\put(5,2){\makebox(0,0){$\bullet$}}
\put(25,2){\makebox(0,0){$\bullet$}}
\put(45,2){\makebox(0,0){$\bullet$}}
\put(5,2){\line(1,0){20}}
\put(25,1){\line(1,0){20}}
\put(25,3){\line(1,0){20}}
\put(35,2){\makebox(0,0){$\langle$}}
\put(5,10){\makebox(0,0){$\scriptstyle 0$}}
\put(25,10){\makebox(0,0){$\scriptstyle k$}}
\put(45,10){\makebox(0,0){$\scriptstyle 0$}}
\end{picture})^{k,k}\,t^k\;=\frac{1+t+t^2}{(1-t)^7},\end{equation}
in agreement with the table above.  Finally, in~\cite[\S(6.3)]{D}, Delong shows
that on~${\mathbb{CP}}_2$ the Killing fields do, indeed, generate the rank $2$
Killing tensors.

There is some general computational machinery~\cite{GZ} for dealing with
so-called `first BGG operators,' such as~(\ref{killing_tensor_operator}), on
homogeneous geometries, such as~${\mathbb{CP}}_n$.  It is not clear how far one
can operate this machinery in computing the dimensions above.  In the
conformal case~\cite{GZconf}, however, Lenka Zalabov\'a informs me that, for
small $n$ and~$k$, their machinery can compute the dimension of the spaces of
rank $k$ {\em \underbar{conformal} Killing tensors\/} on~${\mathbb{CP}}_n$. It 
turns out, for example, that the space of rank 3 conformal Killing tensors on 
${\mathbb{CP}}_2$ also has dimension 119 so all of these correspond to 
ordinary Killing tensors (unlike the round sphere where there are 
always more tensors in the conformal regime).


\section{Killing fields generate}\label{generation}
The aim of this section is to prove Corollary~\ref{killing_fields_generate}, as
a consequence of Theorem~\ref{killing_tensors_on_cpn}.  Thus, we are required
to show the following algebraic statement.
\begin{prop}
The vector space homomorphism 
$$\textstyle\bigodot^k\!\Wedge_\perp^{1,1}({\mathbb{T}})
\longrightarrow\left\{\Sigma\in\begin{picture}(85,14)(0,7)
\put(0,0){\line(0,1){20}}
\put(10,0){\line(0,1){20}}
\put(20,0){\line(0,1){20}}
\put(50,0){\line(0,1){20}}
\put(60,0){\line(0,1){20}}
\put(0,0){\line(1,0){60}}
\put(0,10){\line(1,0){60}}
\put(0,20){\line(1,0){60}}
\put(36,4.5){\makebox(0,0){$\cdots$}}
\put(36,14.5){\makebox(0,0){$\cdots$}}
\put(64,2){\makebox(0,0){$\perp$}}
\put(75,10){\makebox(0,0){$({\mathbb{T}})$}}
\end{picture}\mbox{\Large$\mid$}J\Sigma=0\right\}$$
is surjective, where the subscript $\perp$ means to restrict to the subspace
where all traces with respect to $J^{\alpha\beta}$ vanish.
\end{prop}
\begin{proof}
Both of these vector spaces have natural inner products so it is equivalent to
show the adjoint statement, namely that the dual vector space
homomorphism
$$\textstyle\bigodot^k\!\Wedge_\perp^{1,1}({\mathbb{T}})
\longleftarrow\left\{\Sigma\in\begin{picture}(85,14)(0,7)
\put(0,0){\line(0,1){20}}
\put(10,0){\line(0,1){20}}
\put(20,0){\line(0,1){20}}
\put(50,0){\line(0,1){20}}
\put(60,0){\line(0,1){20}}
\put(0,0){\line(1,0){60}}
\put(0,10){\line(1,0){60}}
\put(0,20){\line(1,0){60}}
\put(36,4.5){\makebox(0,0){$\cdots$}}
\put(36,14.5){\makebox(0,0){$\cdots$}}
\put(64,2){\makebox(0,0){$\perp$}}
\put(75,10){\makebox(0,0){$({\mathbb{T}})$}}
\end{picture}\mbox{\Large$\mid$}J\Sigma=0\right\}$$
is injective.  As already indicated with the notation~$\Wedge_\perp^{1,1}$, if
we complexify both sides then we may capture the statement $J\Sigma=0$ by
decomposing $\Sigma$ into type and insisting that it have the same number of
barred and unbarred indices, namely~$k$.

Firstly, let us ignore the $J^{\alpha\beta}$-trace-zero condition, and 
also the `type $(k,k)$ condition' and try to show that
$$\raisebox{3pt}{$\bigodot^k\!\Wedge^2({\mathbb{T}})\longleftarrow
\begin{picture}(85,14)(0,7)
\put(0,0){\line(0,1){20}}
\put(10,0){\line(0,1){20}}
\put(20,0){\line(0,1){20}}
\put(50,0){\line(0,1){20}}
\put(60,0){\line(0,1){20}}
\put(0,0){\line(1,0){60}}
\put(0,10){\line(1,0){60}}
\put(0,20){\line(1,0){60}}
\put(36,4.5){\makebox(0,0){$\cdots$}}
\put(36,14.5){\makebox(0,0){$\cdots$}}
\put(70,10){\makebox(0,0){$({\mathbb{T}})$}}
\end{picture}$}$$
is injective.  Of course, this is true because this homomorphism is non-zero 
and the right hand side is an irreducible representation of 
${\mathrm{GL}}(2n+2,{\mathbb{R}})$. However, we can also show this by index 
manipulations.  If $k=2$, for example, then on the right hand side we have a 
tensor satisfying 
$$R_{\alpha\beta\gamma\delta}=R_{[\alpha\beta][\gamma\delta]}\enskip
\mbox{and}\enskip R_{[\alpha\beta\gamma]\delta}=0$$
and we are trying to show that it is already in the left hand side, namely that
$$R_{\alpha\beta\gamma\delta}=R_{\gamma\delta\alpha\beta}.$$
This is the standard observation that a tensor satisfying Riemann tensor
symmetries automatically satisfies the interchange symmetry (and this is
usually proved by index manipulations). For $k=3$, we may write 
$$\begin{picture}(30,8)(0,7)
\put(0,0){\line(0,1){20}}
\put(10,0){\line(0,1){20}}
\put(20,0){\line(0,1){20}}
\put(30,0){\line(0,1){20}}
\put(0,0){\line(1,0){30}}
\put(0,10){\line(1,0){30}}
\put(0,20){\line(1,0){30}}
\end{picture}
=\begin{picture}(20,8)(0,7)
\put(0,0){\line(0,1){20}}
\put(10,0){\line(0,1){20}}
\put(20,0){\line(0,1){20}}
\put(0,0){\line(1,0){20}}
\put(0,10){\line(1,0){20}}
\put(0,20){\line(1,0){20}}
\end{picture}\otimes
\begin{picture}(10,8)(0,7)
\put(0,0){\line(0,1){20}}
\put(10,0){\line(0,1){20}}
\put(0,0){\line(1,0){10}}
\put(0,10){\line(1,0){10}}
\put(0,20){\line(1,0){10}}
\end{picture}\;\cap \;
\begin{picture}(10,8)(0,7)
\put(0,0){\line(0,1){20}}
\put(10,0){\line(0,1){20}}
\put(0,0){\line(1,0){10}}
\put(0,10){\line(1,0){10}}
\put(0,20){\line(1,0){10}}
\end{picture}\otimes
\begin{picture}(20,8)(0,7)
\put(0,0){\line(0,1){20}}
\put(10,0){\line(0,1){20}}
\put(20,0){\line(0,1){20}}
\put(0,0){\line(1,0){20}}
\put(0,10){\line(1,0){20}}
\put(0,20){\line(1,0){20}}
\end{picture}$$
and apply this reasoning in the first four and last four indices separately to
conclude that a tensor in this space automatically lies in
$$\textstyle\begin{picture}(10,8)(0,7)
\put(0,0){\line(0,1){20}}
\put(10,0){\line(0,1){20}}
\put(0,0){\line(1,0){10}}
\put(0,10){\line(1,0){10}}
\put(0,20){\line(1,0){10}}
\end{picture}\odot
\begin{picture}(10,8)(0,7)
\put(0,0){\line(0,1){20}}
\put(10,0){\line(0,1){20}}
\put(0,0){\line(1,0){10}}
\put(0,10){\line(1,0){10}}
\put(0,20){\line(1,0){10}}
\end{picture}\otimes
\begin{picture}(10,8)(0,7)
\put(0,0){\line(0,1){20}}
\put(10,0){\line(0,1){20}}
\put(0,0){\line(1,0){10}}
\put(0,10){\line(1,0){10}}
\put(0,20){\line(1,0){10}}
\end{picture}\;\cap\;
\begin{picture}(10,8)(0,7)
\put(0,0){\line(0,1){20}}
\put(10,0){\line(0,1){20}}
\put(0,0){\line(1,0){10}}
\put(0,10){\line(1,0){10}}
\put(0,20){\line(1,0){10}}
\end{picture}\otimes
\begin{picture}(10,8)(0,7)
\put(0,0){\line(0,1){20}}
\put(10,0){\line(0,1){20}}
\put(0,0){\line(1,0){10}}
\put(0,10){\line(1,0){10}}
\put(0,20){\line(1,0){10}}
\end{picture}\odot
\begin{picture}(10,8)(0,7)
\put(0,0){\line(0,1){20}}
\put(10,0){\line(0,1){20}}
\put(0,0){\line(1,0){10}}
\put(0,10){\line(1,0){10}}
\put(0,20){\line(1,0){10}}
\end{picture}
=\bigodot^3\!\Wedge^2,$$
as required.

Now, let us reduce to ${\mathrm{SL}}(n+1,{\mathbb{C}})$ and bring in the `type
$(k,k)$ condition' for the complexified representations.  We aim to show that, 
for example, 
\begin{equation}\label{does_it_inject}\textstyle\begin{picture}(30,8)(0,7)
\put(0,0){\line(0,1){20}}
\put(10,0){\line(0,1){20}}
\put(20,0){\line(0,1){20}}
\put(30,0){\line(0,1){20}}
\put(0,0){\line(1,0){30}}
\put(0,10){\line(1,0){30}}
\put(0,20){\line(1,0){30}}
\end{picture}\,\rule{0pt}{13pt}^{3,3}\longrightarrow
\bigodot^3\Wedge^{1,1}\end{equation}
is injective.  Although the left hand side is no longer irreducible, we may
show this by direct index manipulations as follows (see, e.g.~\cite{CEMN}, for
a formal discussion of barred and unbarred indices). Typically, we find
\begin{equation}\label{index_manipulation}
R_{abc\bar{d}\bar{e}\bar{f}}=-R_{ab\bar{d}c\bar{e}\bar{f}}
=R_{b\bar{d}ac\bar{e}\bar{f}}-R_{a\bar{d}bc\bar{e}\bar{f}}
=R_{b\bar{d}a\bar{e}c\bar{f}}-R_{b\bar{d}c\bar{e}a\bar{f}}
+R_{a\bar{d}c\bar{e}b\bar{f}}-R_{a\bar{d}b\bar{e}c\bar{f}}\end{equation}
and we may similarly rearrange all tensors of type $(3,3)$ in 
$\begin{picture}(21,12)(0,3)
\put(1,1){\line(1,0){18}}
\put(1,7){\line(1,0){18}}
\put(1,13){\line(1,0){18}}
\put(1,1){\line(0,1){12}}
\put(7,1){\line(0,1){12}}
\put(13,1){\line(0,1){12}}
\put(19,1){\line(0,1){12}}
\end{picture}$ 
as a linear combination of tensors, solely of the form
$R_{a\bar{b}c\bar{d}e\bar{f}}$.  These tensors are in
$\bigodot^3\!\Wedge^{1,1}$ and it follows that the kernel of
(\ref{does_it_inject}) vanishes, as required.  The same reasoning evidently
extends to all~$k$.  Finally, it is clear that index manipulations such as
(\ref{index_manipulation}) remain valid if we suppose that the tensors in
question are trace-free with respect to $J^{a\bar{b}}$.
\end{proof}

\section{Poincar\'e series} Recall~(\ref{delong_poincare_series}), taken from
that Delong~\cite[\S6.2]{D}, that
$$\sum_{k\geq 0}\dim
\{\mbox{Killing tensors on ${\mathbb{CP}}_2$ of rank~$k$}\}\,t^k
\;=\frac{1+t+t^2}{(1-t)^7}.$$
For Killing tensors on the round $n$-sphere, Delong~\cite[\S3.4]{D} defines
$$G_n(t)\equiv\sum_{k\geq 0}\dim
\{\mbox{Killing tensors on $S^n$ of rank~$k$}\}\,t^k$$
and shows that
$$G_2(t)=\frac1{(1-t)^3},\enskip G_3(t)=\frac{1+t}{(1-t)^5},\enskip
G_4(t)=\frac{1+3t+t^2}{(1-t)^7},\enskip 
G_5(t)=\frac{1+6t+6t^2+t^3}{(1-t)^9},$$
and so on, where the coefficients in these numerators come from the following 
pyramid
$$\begin{array}{ccccccccccccc}
&&&&&&1\\[-3pt]
&&&&&1&&1\\[-3pt]
&&&&1&&3&&1\\[-3pt]
&&&1&&6&&6&&1\\[-3pt]
&&1&&10&&20&&10&&1\\[-3pt]
&1&&15&&50&&50&&15&&1\\[-3pt]
\raisebox{1pt}{.}\raisebox{4pt}{.}\raisebox{7pt}{.}&&
\raisebox{1pt}{.}\raisebox{4pt}{.}\raisebox{7pt}{.}&&
\raisebox{1pt}{.}\raisebox{4pt}{.}\raisebox{7pt}{.}&&
\vdots&&
\raisebox{7pt}{.}\raisebox{4pt}{.}\raisebox{1pt}{.}&&
\raisebox{7pt}{.}\raisebox{4pt}{.}\raisebox{1pt}{.}&&
\raisebox{7pt}{.}\raisebox{4pt}{.}\raisebox{1pt}{.}
\end{array}$$
sometimes known as the Catalan triangle~\cite{Catalan}.  Strictly speaking, the
rest of this article is conjectural, although the following power series
identities are easily checked (with aid from a computer) regarding the first
few thousand terms. Let us define
$$H_n(t)\equiv\sum_{k\geq 0}\dim
\{\mbox{Killing tensors on ${\mathbb{CP}}_n$ of rank~$k$}\}\,t^k.$$
Delong~\cite[\S6.2]{D} has found $H_2(t)=(1+t+t^2)/(1-t)^7$ and it turns out
that
$$H_3(t)=\frac{1+4t+10t^2+4t^3+t^4}{(1-t)^{11}},\enskip
H_4(t)=\frac{1+9t+45t^2+65t^3+45t^4+9t^5+t^6}{(1-t)^{15}},$$
and so on, where the coefficients in these numerators come from the following 
pyramid
$$\begin{array}{ccccccccccccc}
&&&&&&1\\[-3pt]
&&&&&1&1&1\\[-3pt]
&&&&1&4&10&4&1\\[-3pt]
&&&1&9&45&65&45&9&1\\[-3pt]
&&1&16&136&416&626&416&136&16&1\\[-3pt]
&1&25&325&1700&4550&6202&4550&1700&325&25&1\\[-3pt]
\raisebox{1pt}{.}\raisebox{4pt}{.}\raisebox{7pt}{.}&
\raisebox{1pt}{.}\raisebox{4pt}{.}\raisebox{7pt}{.}&
\raisebox{1pt}{.}\raisebox{4pt}{.}\raisebox{7pt}{.}&
\raisebox{1pt}{.}\raisebox{4pt}{.}\raisebox{7pt}{.}&
\raisebox{1pt}{.}\raisebox{4pt}{.}\raisebox{7pt}{.}&
\raisebox{1pt}{.}\raisebox{4pt}{.}\raisebox{7pt}{.}&
\vdots&
\raisebox{7pt}{.}\raisebox{4pt}{.}\raisebox{1pt}{.}&
\raisebox{7pt}{.}\raisebox{4pt}{.}\raisebox{1pt}{.}&
\raisebox{7pt}{.}\raisebox{4pt}{.}\raisebox{1pt}{.}&
\raisebox{7pt}{.}\raisebox{4pt}{.}\raisebox{1pt}{.}&
\raisebox{7pt}{.}\raisebox{4pt}{.}\raisebox{1pt}{.}&
\raisebox{7pt}{.}\raisebox{4pt}{.}\raisebox{1pt}{.}
\end{array}$$
A formula or algorithm for generating this pyramid directly remains mysterious.
However, it appears that there is an indirect formula as follows.  Let us set
$$L_n(x)\equiv\frac1{2^nn!}\frac{d^n}{dx^n}(x^2-1)^n,$$
the $n^{\mathrm{th}}$ Legendre polynomial:
$$\textstyle L_0(x)=1,\enskip L_1(x)=x,\enskip 
L_2(x)=\frac12(3x^2-1),\enskip
L_3(x)=\frac12(5x^3-3x),\enskip
\ldots$$
and, following Papoulis~\cite{P}, set
$$P_{2k+1}(w)\equiv\frac1{2(k+1)^2}\int_{-1}^{2w-1}\hspace{-15pt}\left[
L_0(x)+3L_1(x)+5L_2(x)+\cdots(2k+1)L_k(x)\right]^2dx,\enskip\mbox{for }k\geq0$$
so that $P_{2k+1}(w)$ is a polynomial of degree $2k+1$, the first few of which 
are
$$P_1(w)=w,\enskip P_3(w)=3w^3-3w^2+w,\enskip P_5(w)=20w^5-40w^4+28w^3-8w^2+w,
\enskip \ldots.$$
Then, it turns out that
$$H_n(t)=-\frac1{t(1-t)^{2n}}P_{2n-1}\Big(\frac{t}{t-1}\Big),\enskip\forall
n\geq 1.$$

\end{document}